\documentclass[review,hidelinks,onefignum,onetabnum]{siamart250211}

\usepackage{backref}
\usepackage{lipsum}
\usepackage{amsfonts}
\usepackage{graphicx}
\usepackage{epstopdf}
\usepackage[numbers]{natbib}
\usepackage{graphicx}%
\usepackage{multirow}%
\usepackage{amsmath,amssymb,amsfonts,bm}%
\usepackage{mathrsfs}%
\usepackage[title]{appendix}%
\usepackage{xcolor}%
\usepackage{textcomp}%
\usepackage{manyfoot}%
\usepackage{booktabs}%
\usepackage{algorithm}%
\usepackage{algorithmicx}%
\usepackage{algpseudocode}%
\usepackage{listings}%
\usepackage{amsgen}

\usepackage{hyperref,enumerate,subfigure}
\usepackage{color,natbib}
\usepackage{epsfig,epstopdf,bm, caption2}
\usepackage{wrapfig}
\ifpdf
\DeclareGraphicsExtensions{.eps,.pdf,.png,.jpg}
\else
\DeclareGraphicsExtensions{.eps}
\fi

\newsiamremark{remark}{Remark}
\newsiamremark{hypothesis}{Hypothesis}
\newsiamthm{assumption}{Assumption}
\crefname{hypothesis}{Hypothesis}{Hypotheses}
\newsiamthm{claim}{Claim}
\newsiamremark{fact}{Fact}
\crefname{fact}{Fact}{Facts}
\newsiamthm{mytheo}{Theorem}
\newsiamthm{prop}{Proposition}
\nolinenumbers

\headers{Time-relaxation low-regularity integrators for NLS}{H. Li, X. Li, K. Schratz, and B. Wang}

\title{Time-relaxation structure-preserving explicit low-regularity integrators for the nonlinear Schrödinger equation\thanks{Submitted to the editors DATE.
		\funding{The first and third authors are supported by the European Research Council (ERC) under the European Union’s Horizon 2020 research and innovation program (grant agreement No. 850941). The forth author is supported by NSFC 12371403.}}}

\author{Hang Li\thanks{Laboratoire Jacques-Louis Lions, Sorbonne Universit\'e, Paris 75005, France (\email{hang.li.1@sorbonne-universite.fr}, \email{katharina.schratz@sorbonne-universite.fr}).}
	\and Xicui Li\thanks{Corresponding author. Laboratoire Jacques-Louis Lions, Sorbonne Universit\'e, Paris 75005, France; School of Mathematics and Statistics, Xi’an Jiaotong University, Xi’an 710049, People’s Republic of China
		(\email{lixicui@stu.xjtu.edu.cn}).}
	\and Katharina Schratz\footnotemark[2]\and Bin Wang\thanks{School of Mathematics and Statistics, Xi’an Jiaotong University, Xi’an 710049, People’s Republic of China
		(\email{wangbinmaths@xjtu.edu.cn}).}}

\usepackage{amsopn}

\ifpdf
\hypersetup{
  pdftitle={Time-relaxation low-regularity integrators for NLS},
  pdfauthor={H. Li, X. Li, K. Schratz, and B. Wang}
}
\fi

\newcommand{\abs}[1]{\lvert#1\rvert}

\def\<{\left\langle}
\def\>{\right\rangle}

\newcommand{\norm}[1]{\left\Vert#1\right\Vert}

\newcommand{\be}{\begin{equation}}
	\newcommand{\ee}{\end{equation}}
\newcommand{\e}{\mathrm{e}}
\begin{document}

\maketitle
\begin{abstract}
We propose and rigorously analyze a novel family of explicit low-regularity exponential integrators for the nonlinear Schrödinger (NLS) equation, based on a time-relaxation framework. The methods combine a resonance-based scheme for the twisted variable with a dynamically adjusted relaxation parameter that guarantees exact mass conservation. Unlike existing symmetric or structure-preserving low-regularity integrators, which are typically implicit and computationally expensive, the proposed methods are fully explicit, mass-conserving, and well-suited for solutions with low regularity. Furthermore, the schemes can be naturally extended to a broad class of evolution equations exhibiting the structure of strongly continuous contraction semigroups. Numerical results demonstrate the accuracy, robustness, and excellent long-time behavior of the methods under low-regularity conditions.
\end{abstract}



\section{Introduction}
In this work, we focus on the numerical approximation of solutions to the nonlinear Schrödinger (NLS) equation given by:
\begin{equation}\label{general_NLS}
	\begin{cases}
		i \partial_t u(t,\mathbf{x}) +\Delta u(t,\mathbf{x})= \lambda  \abs{u(t,\mathbf{x})}^{2p}u(t,\mathbf{x}), \ \ &(t,\mathbf{x}) \in [0,T] \times \mathbb{T}^d, \\
		u(0,\mathbf{x})= u_0(\mathbf{x}),&\mathbf{x} \in \mathbb{T}^d,
	\end{cases}
\end{equation}
where $\lambda\in \mathbb{R}$, $p, d \in \mathbb{N}_{+}$, and $\mathbb{T}^d = (\mathbb{R}/(2\pi\mathbb{Z}))^d$ is the $d$-dimensional torus. The NLS equation \eqref{general_NLS} admits two fundamental conservation laws:
\paragraph{Mass conservation}
The solution conserves the mass (or the quadratic first integral):
\begin{equation}\label{mass_conservation}
	\mathcal{M}(u(t)) := \int_{\mathbb{T}^d} |u(t,\mathbf{x})|^2 \, \mathrm{d}\mathbf{x} \equiv \int_{\mathbb{T}^d} |u_0(\mathbf{x})|^2 \, \mathrm{d}\mathbf{x} = \mathcal{M}(u_0), \quad t \in [0,T],
\end{equation}
which is equivalent to
\begin{equation*}
	\|u(t)\|_{L^2(\mathbb{T}^d)} \equiv\|u_0\|_{L^2(\mathbb{T}^d)}, \quad t \in [0,T].
\end{equation*}
\paragraph{Energy conservation}
The solution also conserves the Hamiltonian:
\begin{equation}\label{energy_conservation}
	\mathcal{H}(u(t)) := \int_{\mathbb{T}^d} \left[ |\nabla u(t,\mathbf{x})|^2 + \frac{\lambda}{p+1} |u(t,\mathbf{x})|^{2p+2} \right] \, \mathrm{d}\mathbf{x} \equiv \mathcal{H}(u_0), \quad t \in [0,T].
\end{equation}

The NLS equation arises as a fundamental model in numerous areas of physics, such as quantum physics and chemistry, Bose–Einstein condensation (BEC), laser propagation, plasma and particle dynamics~\cite{bao2012mathematical,erdhos2007derivation,sulem1999nonlinear}, and also in nonlinear optics and hydrodynamics~\cite{sulem1999nonlinear,tai1986observation,thomas2012nonlinear}.
Numerically, many classical methods have been developed for solving the NLS equation, including finite difference methods \cite{akrivis1993finite,antoine2013computational,bao2013optimal}, exponential integrators \cite{bao2014uniform,Celledoni2008,hochbruck2010exponential}, and time splitting methods \cite{bao2023improved,bao2003numerical,besse2002order,Lubich2008}. However, when solutions lack sufficient regularity, these classical numerical schemes suffer from loss of stability and accuracy, and may even exhibit order reduction. In 2018, Ostermann \& Schratz~\cite{ostermann_schratz2018low} developed a resonance-based framework for constructing low-regularity exponential integrators applicable to nonsmooth solutions. This foundational work has since inspired the development of higher-order schemes~\cite{bruned_schratz_2022,ostermann2022second}, integrators with even weaker regularity assumptions~\cite{cao2024new,li2021fully,ostermann2022fully}, and extensions to more general boundary settings~\cite{bronsard2023error,bai2022constructive,rousset2021general}. However, none of the aforementioned integrators are structure-preserving, as they disrupt the underlying geometric structure~\cite{Celledoni2008}. Building on the time-reversibility of the equation, Alama Bronsard~\cite{alama2023symmetric} and Feng et al.~\cite{feng_maierhofer_schratz_2023} proposed symmetric low-regularity schemes that approximately preserve mass and energy over long time intervals, while Maierhofer \& Schratz~\cite{Maierhofer2022} developed symplectic integrators that exactly preserve quadratic invariants such as mass. However, in all of the aforementioned works on symmetric low-regularity integrators, introducing symmetry necessarily led to implicit schemes, thereby incurring the additional computational cost of solving a nonlinear system at each time step.
Inspired by the works of Bao {\& Wang}~\cite{bao2024explicit} and Jahnke et al.~\cite{jahnke2023numerical, jahnke2018adiabatic, jahnke2019adiabatic}, recent advances in exponential integrator design have attempted to address this issue by constructing fully explicit two-step schemes ~\cite{feng2025explicit} that retain symmetry, exhibit favorable long-time behavior and approximately preserve structural properties. While no explicit low-regularity integrator currently exists that rigorously preserves such invariants, we bridge this gap by proposing and rigorously analyzing the first explicit low-regularity integrator based on a \emph{time-relaxation} framework, which exactly conserves mass and exhibits favorable long-time behavior. This framework adapts ideas from time-relaxation techniques used in energy-preserving ODE schemes~\cite{Ketcheson2019,Ranocha2020,Ranocha2022} and extends them to the low-regularity PDE setting. Moreover, due to its general formulation, the proposed method can be naturally extended to a wide class of evolution equations possessing the structure of strongly continuous contraction semigroups~\cite{rousset2021general}, including the Gross–Pitaevskii (GP) equation, which arises from mean-field approximations of many-body quantum systems and is particularly important in modeling Bose–Einstein condensation~\cite{bronsard2023error,bao2012mathematical}.

The construction of the relaxation low-regularity integrators proceed as follows:
\begin{enumerate}
	\item We first introduce the twisted variable \( v(t) := \e^{-it\Delta} u(t) \), which removes the dominant linear oscillation from the solution. Based on this transformation, we introduce an \( m \)-th order (\( m \geq 2 \)) resonance-based low-regularity scheme, written as
	\begin{align}
		v^{n+1} := v^n + \widehat{\Phi}_{t_n}^\tau(v^n), \quad v^0 = v(0), \quad n \geq 0,
	\end{align}
	where \( \widehat{\Phi}_{t_n}^\tau(v^n) \) is derived from the resonance structure at time \( t_n \).
	
	\item We then apply a time-relaxation technique to the twisted variable by introducing a sequence of relaxation parameters \( \gamma_n \), leading to the modified scheme:
	\begin{equation*}
		v_\gamma^{n+1} := v_\gamma^n + \gamma_n \widehat{\Phi}_{\tilde{t}_n}^\tau(v_\gamma^n), \quad n \geq 0,
	\end{equation*}
	where \( v_\gamma^0 = u_0 \), \( \tilde{t}_0 = t_0 = 0 \), and \( \tilde{t}_{n+1} := \tilde{t}_n + \gamma_n \tau \), with \( \gamma_n \) chosen to ensure exact mass conservation. Moreover, we prove that
	\[
	|\gamma_n - 1| \leq C \tau^{m-1},
	\]
	which ensures that the relaxation does not degrade the accuracy and allows for a local error estimate under low regularity assumptions.
	
	\item Combining this with a stability analysis yields the global error bound for the proposed  low-regularity integrators.
\end{enumerate}

\begin{remark}
	The use of the twisted variable \( v \) instead of the original solution \( u \) in the construction of the time-relaxation scheme is motivated by the need to preserve the low-regularity local error after applying relaxation—or equivalently, to retain the desired convergence order under low regularity assumptions (cf.~Subsection~\ref{appen}). Additionally, the transformation is \( L^2 \)-isometric, i.e., \( \|v(t)\|_{L^2} = \|u(t)\|_{L^2}, \ t\geq 0 \), thereby ensuring mass conservation at the level of the transformed variable.
\end{remark}

The proposed relaxation-based low-regularity integrators enjoys several advantages:
\begin{itemize}
	\item Fully explicit and easy to implement;
	\item Suitable for solutions with low regularity;
	\item Exactly preserve mass and maintain favorable long-time behavior through structure-preserving design.
\end{itemize}

The remainder of this paper is organized as follows. In Section~\ref{construction}, we present the construction of the {numerical methods}, starting from a class of resonance-based integrators formulated in terms of the twisted variable. In Section~\ref{CA-SP}, we demonstrate that exact mass conservation can be achieved via a carefully chosen relaxation parameter followed by a local error analysis under low-regularity assumptions and a global convergence result obtained by establishing the stability of the scheme. Numerical results illustrate the accuracy, efficiency, structure preservation, and long-time behavior of the methods are provided in Section~\ref{ne}. Finally, we conclude the paper and outline future research directions in Section~\ref{con}.

\section{Relaxation low regularity integrators}\label{construction}
In this section, we will formulate a new class of structure-preserving and low regularity integrators, which are constructed by using the idea of relaxation to modify low regularity schemes for nonlinear Schr\"{o}dinger equation. To begin with, we give some notations and techniques, which we will work with throughout this paper.
\subsection{Preliminaries}
We start with some important notations, which are introduced for the convenience of the numerical analysis later. The notation \(A\lesssim B\) or \(B\gtrsim A\) denotes \(A\leq CB\) with some generic constant \(C>0\).  Furthermore, we use the symbol \(X=\mathcal{O}(Y)\) to represent any quantity \(X\) such that \(\abs{X}\lesssim \abs{Y}\) with the Euclidean norm $\abs{\cdot}$. {We underline the dependence of constants as follows: \(C(\cdot),\,C(\cdot,\cdot)\), or similar notations denote generic positive continuous functions, which remain bounded whenever their arguments are restricted to a bounded domain.}

We continue with the notation
\[\langle \boldsymbol{\xi}\rangle=\sqrt{1+\boldsymbol{\xi}\cdot\boldsymbol{\xi}}=(1+\abs{\boldsymbol{\xi}}^2)^{\frac{1}{2}},\ \ \text{for}\  \boldsymbol{\xi}=(\xi_1,\xi_2,...,\xi_d)\in\mathbb{Z}^d.\] Then the Fourier transform of a function on \(\mathbb{T}^d\) is defined by
\[\mathcal{F}(f)(\boldsymbol{\xi})=\hat{f}(\boldsymbol{\xi})=\frac{1}{(2\pi)^d}\int_{\mathbb{T}^d}\e^{-i\mathbf{x}\cdot\boldsymbol{\xi}}f(\mathbf{x})\mathrm{d}\mathbf{x},\]
and hence the Fourier inversion formula reads as
\[f(\mathbf{x})=\int \e^{i\mathbf{x}\cdot\boldsymbol{\xi}}\hat{f}(\boldsymbol{\xi})\,(\mathrm{d}\boldsymbol{\xi}),\]
where \((\mathrm{d}\boldsymbol{\xi})\) is the normalized counting measure on \(\mathbb{Z}^d\) such that
\[\int a(\boldsymbol{\xi})(\mathrm{d}\boldsymbol{\xi})=\sum_{\boldsymbol{\xi}\in\mathbb{Z}^d}a(\boldsymbol{\xi}).\]
Subsequently we have the following usual properties of the Fourier transform:
\begin{align*}
	&\norm{f}_{L^2(\mathbb{T}^d)}=(2\pi)^{\frac{d}{2}}\norm{\hat{f}}_{l^2(\mathbb{Z}^d)}\quad \text{(Plancherel)};\\
	&\langle f, g \rangle 
	= \int_{\mathbb{T}^d} f(\mathbf{x})\, \overline{g(\mathbf{x})} \, \mathrm{d}\mathbf{x} 
	= (2\pi)^d \int \hat{f}(\boldsymbol{\xi})\, \overline{\hat{g}(\boldsymbol{\xi})} \, (\mathrm{d}\boldsymbol{\xi})
	\quad \text{(Parseval)};\\
	&\widehat{(f g)}(\boldsymbol{\xi}) 
	= \int \hat{f}(\boldsymbol{\xi} - \boldsymbol{\eta})\, \hat{g}(\boldsymbol{\eta})\, (\mathrm{d}\boldsymbol{\eta})
	\quad \text{(convolution)}.
\end{align*}
The Sobolev space \(H^s(\mathbb{T}
^d)\ \text{for}\  s\geq 0\) equips with the equivalent norm
\[\norm{f}_{H^s(\mathbb{T}^d)}=\norm{J^sf}_{L^2(\mathbb{T}^d)}=(2\pi)^{\frac{d}{2}}\norm{\langle\boldsymbol{\xi}\rangle^s\hat{f}(\boldsymbol{\xi})}_{l^2(\mathbb{Z}^d)},\]
with the operator \(J^s=(1-\Delta)^{\frac{s}{2}}\), \((s\geq 0)\). In the rest of this paper, to save space and typing efforts, we shall use the notations \(\norm{\,\cdot\,}_{H^s}\) and \(\norm{\,\cdot\,}_{L^2}\) to denote the norm of \(H^s(\mathbb{T}^d)\) and \(L^2(\mathbb{T}^d)\), respectively.

Throughout the error analysis we will repeatedly use the following well-known bilinear estimates.
\begin{lemma}\label{bilinear_estimates}
For any \(r>d/2\), \(f,\,g\in H^r\), we have
\begin{equation*}
\norm{fg}_{H^r}\lesssim \norm{f}_{H^r} \norm{g}_{H^r}.
\end{equation*}
\end{lemma}
Next we define two functions
\begin{equation}\label{varphi_psi}
	\varphi_1(z)=\int_{0}^{1}\e^{zs}\mathrm{d}s=\begin{cases}
		\frac{\e^z-1}{z},&z\neq 0,\\
		1,&z=0,
	\end{cases} \quad\varphi_2(z)=\int_{0}^{1}s\e^{zs}\mathrm{d}s=\begin{cases}
		\frac{z\e^z-\e^z+1}{z^2},&z\neq 0,\\
		\frac{1}{2},&z=0.
	\end{cases}
\end{equation}
Then it can be easily obtained that \(\abs{\varphi_1(iz)}\leq 1\) and \(\abs{\varphi_2(iz)}\leq \frac{1}{2}\). Furthermore, based on the inequality \(\abs{\e^{ix}-1}\leq \abs{x}\), we have the following inequalities which can be used for error estimates.
\begin{lemma}\label{inequalities}
	For all \(x,\,y\in \mathbb{R}\), it holds that
	\begin{align*}
		&\abs{\varphi_1(ix)-\varphi_1(iy)} \leq \frac{
			1}{2}\abs{x-y}, \qquad \quad\,
		\abs{\varphi_2(ix)-\varphi_2(iy)} \leq \frac{
			1}{3}\abs{x-y}, \\
		&\abs{\e^{ix}\varphi_1(ix)-\e^{iy}\varphi_1(iy)} \leq 2\abs{x-y}, \quad 
		\abs{\e^{ix}\varphi_2(ix)-\e^{iy}\varphi_2(iy)} \leq \abs{x-y}.
	\end{align*}
\end{lemma}
\begin{proof}
	From the definition of functions \(\varphi_1\) and \(\varphi_2\) in \eqref{varphi_psi}, we have
	\begin{align*}
		\abs{\varphi_1(ix)-\varphi_1(iy)}&\leq \int_{0}^{1}\abs{\e^{isx}-\e^{isy}}\mathrm{d}s\leq \int_{0}^{1}\abs{s(x-y)}\mathrm{d}s \leq \frac{
			1}{2}\abs{x-y},\\
		\abs{\varphi_2(ix)-\varphi_2(iy)}&\leq \int_{0}^{1}s\abs{\e^{isx}-\e^{isy}}\mathrm{d}s\leq \int_{0}^{1}\abs{s^2(x-y)}\mathrm{d}s \leq \frac{
			1}{3}\abs{x-y}.
	\end{align*}
	According to above results, using triangle inequality and the facts \(\abs{\varphi_1(iz)}\leq 1\), \(\abs{\varphi_2(iz)}\leq \frac{1}{2}\), we have
	\begin{align*}
		\abs{\e^{ix}\varphi_1(ix)-\e^{iy}\varphi_1(iy)}&\leq \abs{(\e^{ix}-\e^{iy})\varphi_1(ix)}+\abs{\e^{iy}(\varphi_1(ix)-\varphi_1(iy))}\leq 2\abs{x-y},\\
		\abs{\e^{ix}\varphi_2(ix)-\e^{iy}\varphi_2(iy)}&\leq \abs{(\e^{ix}-\e^{iy})\varphi_2(ix)}+\abs{\e^{iy}(\varphi_2(ix)-\varphi_2(iy))}\leq \abs{x-y}.
	\end{align*}
	The proof is completed.
\end{proof}
\subsection{Construction of the methods}\label{general construction}
We now provide the underlying idea behind building a new class of structure-preserving LRIs, which can preserve the geometry structure of the NLS equation \eqref{general_NLS} under rough initial data.

We employ the \textit{twisted} variable
\[
v(t,\mathbf{x}) = \exp(-i t \Delta) u(t,\mathbf{x}),
\]
which satisfies the following initial value problem that is equivalent to \eqref{general_NLS}:
\begin{equation}\label{general_NLS_v}
	\begin{cases}
		i \partial_t v(t,\mathbf{x}) = \lambda \e^{-i t\Delta} \Big(\big(  \e^{i t \Delta} v(t,\mathbf{x})\big)^{p+1}\big( \e^{-i t \Delta} \overline{v(t,\mathbf{x})}\big)^p\Big), \quad &(t,\mathbf{x}) \in [0,T] \times \mathbb{T}^d, \\
		v(0,\mathbf{x}) = u_0(\mathbf{x}),\quad &\mathbf{x} \in \mathbb{T}^d.
	\end{cases}
\end{equation}
Then it turns out that the \(L^2\)-norm of \(v\) is still a conserved quantity since the operator \( \exp(-i t \Delta)\) is an isometry on \(L^2\) space, namely,
\begin{equation}
	\norm{v(t)}_{L^2}=\norm{u(t)}_{L^2}=\norm{u_0}_{L^2}=\norm{v(0)}_{L^2}.
\end{equation}
Furthermore, the mild solution at \(t=t_n+\sigma\) for \eqref{general_NLS_v} can be written in the following form (using the short notation \(v(t)=v(t,\mathbf{x})\))
\begin{align*}
	v(t_n &+ \sigma)=v(t_n)\\
	& - i\lambda \int_0^{\sigma} \e^{-i (t_n + s) \Delta} 
	\left( \left( \e^{i(t_n + s)\Delta} v(t_n + s)\right) ^{p+1}\left( 
	\e^{-i(t_n + s)\Delta} \overline{v(t_n + s)} \right)^p\right)   \mathrm{d} s.
\end{align*}
Therefore, the key step of constructing the numerical schemes for equation \eqref{general_NLS_v} is to properly approximate the integral: \[I_\sigma:=\int_0^{\sigma} \e^{-i (t_n + s) \Delta} \left( \left( \e^{i(t_n + s)\Delta} v(t_n + s)\right) ^{p+1}\left( 
\e^{-i(t_n + s)\Delta} \overline{v(t_n + s)} \right)^p\right)   \mathrm{d} s.\]
In this case, we can claim that the general numerical scheme for solving  equation \eqref{general_NLS_v} reads as the following formula: 
\begin{equation}\label{general_v_LRI}
	v^{n+1}:=v^n+\widehat{\Phi}_{t_n}^\tau(v^n),\quad v^0=v(0), \quad n\geq 0,
\end{equation}
where \(v^n\) is the approximation of \(v(t_n)\) with \(t_n:=n\tau\) and \(\tau\) is time step size. Here and after, we focus only on fully explicit LRIs in the form of \eqref{general_v_LRI} which possess high computational efficiency and are \(m\)th-order \((m\geq 2)\).
Twisting it back to \(u\), we easily obtain a one-step numerical scheme for \eqref{general_NLS}:
\begin{equation}\label{general_u_LRI}
	u^{n+1}:=\Psi^\tau(u^n)=\e^{i \tau \Delta}u^n+\e^{i t_{n+1} \Delta}\widehat{\Phi}_{t_n}^\tau(\e^{-i t_n \Delta}u^n),
\end{equation}
which owns the same regularity assumption as \eqref{general_v_LRI}.

In order to formulate a general frame of new methods, we suppose the following assumptions hold.
\begin{assumption}\label{Assumption}
	Let \(\Phi_{t}^\tau(f)=f+\widehat{\Phi}_{t}^\tau(f)\) for \(f\in H^r\) \((r\geq0)\).
	
	\textnormal{(i)} \emph{\underline{The method is unconditionally stable,}} i.e., there is a continuous function \(L\): \(\mathbb{R}_{\geq 0}\times \mathbb{R}_{\geq 0} \rightarrow \mathbb{R}_{\geq 0}\) such that for \(f,g \in H^r\), we have 
	\begin{equation*}
		\norm{\widehat{\Phi}_{t}^\tau(f)-\widehat{\Phi}_{t}^\tau(g)}_{H^r}\leq\tau L\,(\norm{f}_{H^r}, \norm{g}_{H^r}) \norm{f-g}_{H^r}.
	\end{equation*}
	
	\textnormal{(ii)} \emph{\underline{Local error is of order \(m+1\),}} i.e., there exists a continuous function \(M:\ \mathbb{R}_{\geq 0}\rightarrow  \mathbb{R}_{\geq 0}\) such that for \(v\in H^{r+\beta}\), the local truncation error \(\mathcal{R}_n^\tau:=v(t_n+\tau)-\widehat{\Phi}_{t_n}^\tau(v(t_n))\) satisfies 
	\begin{equation*}
		\lVert \mathcal{R}_n^\tau\rVert_{H^r}\leq \tau^{m+1}M\Big(\sup_{t\in[t_n,t_{n+1}]}\lVert v(t)\rVert_{H^{r+\beta}} \Big),\quad m\geq 2, \ \beta\geq 0,\ {r+\beta\geq2}. 
	\end{equation*}
\end{assumption}

{\begin{remark}
    Assumption~\ref{Assumption} states that $\Phi_t^\tau$ is a low-regularity integrator that achieves $m$th-order convergence in $H^r$ for solutions in $H^{r+\beta}$, where $r \geq 0$, $m \geq 2$, $\beta \geq 0$, and $r + \beta \geq 2$. The parameter $\beta$ indicates the regularity required of the solution to achieve the desired convergence rate in $H^r$; ideally, $\beta = 0$ implies no loss. 
\end{remark}}

In the following, we give the details of construction for structure-preserving low regularity methods.

\textbf{(I) {The quadratic first integral evolution by LRIs.}} We apply the LRIs in the form of \eqref{general_v_LRI} to solve \eqref{general_NLS_v}. Then the change in the quadratic first integral between one step and the next is
\begin{align*}
	\norm{v^{n+1}}^2_{L^2}-\norm{v^n}^2_{L^2}
	=&\,\int_{\mathbb{T}^d}\abs{v^n+\widehat{\Phi}_{t_n}^\tau(v^n)}^2-\abs{v^n}^2\,\mathrm{d}\mathbf{x}\\
	=&2\operatorname{Re}\left(\langle v^n,\widehat{\Phi}_{t_n}^\tau(v^n)\rangle \right) +\norm{\widehat{\Phi}_{t_n}^\tau(v^n)}^2_{L^2}.
\end{align*}
Thus the equality \(\norm{v^{n+1}}^2_{L^2}=\norm{v^n}^2_{L^2}\) holds if and only if \[2\operatorname{Re}\left(\langle v^n,\widehat{\Phi}_{t_n}^\tau(v^n)\rangle \right) +\norm{\widehat{\Phi}_{t_n}^\tau(v^n)}^2_{L^2}=0.\] 

For the existed LRIs, the above equality only can be satisfied for implicit symplectic LRIs \cite{Maierhofer2022}. However, it's almost impossible for explicit LRIs to exactly preserve the quadratic first integral until now. Accordingly, constructing fully explicit integrators that preserve geometric structure and accommodate low regularity initial data is of great significance.

\textbf{(II) {Relaxation low regularity integrators}}. Adding a non-zero relaxation parameter \(\gamma_n\) \cite{Ketcheson2019,Ranocha2020,Ranocha2022} to the scheme \eqref{general_v_LRI}, then a class of relaxation low regularity integrators read as: 
{\begin{equation}\label{relax_general_v_LRI}
	v_\gamma^{n+1}:=\Phi_\gamma^n(v_\gamma^n),\quad \Phi_\gamma^n(f):=f+\gamma_n\widehat{\Phi}_{\tilde{t}_n}^\tau(f), \quad n\geq 0,
\end{equation}}
where \(v_\gamma^{n+1}\) is the approximation of \(v(\tilde{t}_{n+1})\) and \(\tilde{t}_{n+1}:=\tilde{t}_{n}+\gamma_n\tau\) with \(\tilde{t}_0=t_0=0\), \(v_\gamma^0=u_0\). The parameter \(\gamma_n\) is determined by the equation:
\begin{equation*}
	\norm{v_\gamma^{n+1}}^2_{L^2}=\norm{v_\gamma^n}^2_{L^2},\quad n\geq 0.
\end{equation*}

As a matter of fact, it's easy to calculate the value of \(\gamma_n\) in the following manners: 
\begin{align*}
	\,\norm{v_\gamma^{n+1}}^2_{L^2}-\norm{v_\gamma^n}^2_{L^2}
	=&\,\int_{\mathbb{T}^d}\abs{v_\gamma^n+\gamma_n\widehat{\Phi}_{\tilde{t}_n}^\tau(v_\gamma^n)}^2-\abs{v_\gamma^n}^2\,\mathrm{d}\mathbf{x}\\
	=&\,2\gamma_n\operatorname{Re}\left(\big\langle v_\gamma^n,\widehat{\Phi}_{\tilde{t}_n}^\tau(v_\gamma^n)\big\rangle \right) +\gamma_n^2\norm{\widehat{\Phi}_{\tilde{t}_n}^\tau(v_\gamma^n)}^2_{L^2}.
\end{align*}
Then we set the last equality to zero to obtain 
\[\gamma_n=-\frac{2\operatorname{Re}\left(\big\langle v_\gamma^n,\widehat{\Phi}_{\tilde{t}_n}^\tau(v_\gamma^n)\big\rangle \right)}{\norm{\widehat{\Phi}_{\tilde{t}_n}^\tau(v_\gamma^n)}^2_{L^2}},\quad \norm{\widehat{\Phi}_{\tilde{t}_n}^\tau(v_\gamma^n)}_{L^2}\neq0.\]
In addition, if \(\norm{\widehat{\Phi}_{\tilde{t}_n}^\tau(v_\gamma^n)}_{L^2}=0\), then \(\widehat{\Phi}_{\tilde{t}_n}^\tau(v_\gamma^n)=0\), i.e., \(\tilde{v}^{n+1}:=v_\gamma^n+\widehat{\Phi}_{\tilde{t}_n}^\tau(v_\gamma^n)=v_\gamma^n\) is the approximation of \(v(t)\) at \(t=\tilde{t}_n+\tau\). In this case, we can achieve conservation by taking simply \(\gamma_n=1\). Thus the parameter \(\gamma_n\) can be defined explicitly as
\begin{equation}\label{gamma_n}
	\gamma_n=\begin{cases}
		1-\frac{\norm{\Phi_{\tilde{t}_n}^\tau(v_\gamma^n)}^2_{L^2}-\norm{v_\gamma^n}^2_{L^2}}{\norm{\widehat{\Phi}_{\tilde{t}_n}^\tau(v_\gamma^n)}^2_{L^2}},\quad & \norm{\widehat{\Phi}_{\tilde{t}_n}^\tau(v_\gamma^n)}_{L^2}\neq0,\\
		1,\quad &\norm{\widehat{\Phi}_{\tilde{t}_n}^\tau(v_\gamma^n)}_{L^2}=0,
	\end{cases}
\end{equation}
which shows the well-posedness of new obtained methods.

Subsequently, we substitute its value into the formula \eqref{relax_general_v_LRI} to get the numerical solution which can preserve the invariant \(\norm{v(t)}_{L^2}.\)

Finally,  we can trun it back to \(u\) by
\(u_\gamma^n=\e^{i\tilde{t}_n\Delta}v_\gamma^n,\  n\geq 0, \)
\begin{equation}\label{RLRIs-v}
	u_\gamma^{n+1}=\e^{i \gamma_n\tau \Delta}u_\gamma^n+\gamma_n\e^{i \tilde{t}_{n+1} \Delta}\widehat{\Phi}_{\tilde{t}_n}^\tau(\e^{-i \tilde{t}_n \Delta}u_\gamma^n),
\end{equation}
which naturally conserves the quantity \(\norm{u(t)}_{L^2}\) as \(\norm{u_\gamma^{n+1}}_{L^2}=\norm{u_\gamma^n}_{L^2}\), where we use the fact that \( \exp(i t \Delta)\) is an isometry on \(L^2\) again. {The numerical flow~\eqref{RLRIs-v} is denoted `RLRIs-v' (to be contrasted with ‘RLRI-u’ later), where the suffixes `-v' and `-u' refer to the distinct variable-relaxation employed in their construction.}

For smooth initial data and numerical flow \(\Psi^\tau(u)\) of order-\(m\), it has been proved that we can find a relaxation parameter \(\gamma_n=1+\mathcal{O}(\tau^{m-1})\) satisfies \(\|u^n+\gamma_n(\Psi^\tau(u^n)-u^n)\|_{L^2}=\norm{u^n}_{L^2}\), and it maintains the convergence order of the original methods \(u^{n+1}=\Psi^\tau(u^n)\) \cite{Ketcheson2019,Ranocha2020,Ranocha2022}. However, if we add the parameter to the LRIs of \(u\), it may lead to an order reduction {(especially for \(\beta<4\))}, which will be proved in Subsection \ref{appen}. Hence, in order to avoid the order reduction, we only pay attention to the RLRIs-v \eqref{RLRIs-v} which is constructed based on the relaxation of twisted variable \(v\).

In what follows, we introduce a second order low regularity integrator for solving a classical equation: the cubic Schrödinger equation, and then extend the scheme to the general NLS equation \eqref{general_NLS}. 
\subsubsection{RLRI for cubic NLS}
We first consider the nonlinear Schr\"{o}dinger equation with cubic nonlinearity on a \(d\ (=1,2,3)\) dimensional torus:
\begin{equation}\label{eq_u_cubic_NLS}
	\begin{cases}
		i\partial_t u(t,\mathbf{x})=-\Delta u(t,\mathbf{x})+\lambda \abs{u(t,\mathbf{x})}^2 u(t,\mathbf{x}),\ \ &(t,\mathbf{x})\in [0,T]\times \mathbb{T}^d,\\
		u(0,\mathbf{x})=u_0(\mathbf{x}),&\mathbf{x}\in\mathbb{T}^d,
	\end{cases}
\end{equation}
where \(\lambda=\pm1\) and \(u(t,\mathbf{x})\in \mathbb{C}\). Clearly, this cubic NLS is in the form of \eqref{general_NLS} with \(p=1\).

Introducing the following second order low regularity integrator developed in \cite{bruned_schratz_2022,ostermann2022second} for cubic NLS as an example:
\begin{equation}\label{u_NLS}
	u^{n+1}:=\Psi^\tau_{1,2}(u^n),\quad n\geq 0,\quad u^0=u_0,
\end{equation}
where
\begin{equation}\label{Psi_1}
	\begin{aligned}
		\Psi^\tau_{1,2}(f)=&\,\e^{i\tau\Delta}f-i\lambda\tau \e^{i\tau\Delta}\big[(f)^2\cdot \left( \varphi_1(-2i\tau\Delta)-\varphi_2(-2i\tau\Delta)\right)\bar{f} \big] \\
		&\,-i\lambda\tau (\e^{i\tau\Delta}f)^2\cdot\big(\e^{i\tau\Delta}\varphi_2(-2i\tau\Delta)\bar{f} \big)-\lambda^2\frac{\tau^2}{2}\e^{i\tau\Delta}\big[\abs{f}^4 f\big],
	\end{aligned}
\end{equation}
corresponding to the scheme of its twisted variable \(v:=\e^{-it\Delta}u\):
\begin{equation}\label{v_cubic_NLS}
	v^{n+1}:=v^n+\psi_{t_n}^\tau(v^n)
\end{equation}
equipped with
\begin{equation}\label{psi_t}
	\begin{aligned}
		\psi_{t}^\tau(f)=&\,-i\lambda\tau \e^{-it\Delta}\left[(\e^{it\Delta}f)^2\cdot \left( \varphi_1(-2i\tau\Delta)-\varphi_2(-2i\tau\Delta)\right)\e^{-it\Delta}\bar{f} \right] \\
		&-i\lambda\tau \e^{-i(t+\tau)\Delta}\left[ (\e^{i(t+\tau)\Delta}f)^2\cdot\varphi_2(-2i\tau\Delta)(\e^{-i(t-\tau)\Delta}\bar{f}) \right]\\
		&-\lambda^2\frac{\tau^2}{2}\e^{-it\Delta}\left[\abs{\e^{it\Delta}f}^4 \e^{it\Delta}f\right].
	\end{aligned}
\end{equation}
Then we can apply relaxation technique to the scheme as \(v_\gamma^{n+1}:=v^n+\gamma_n\psi_{\tilde{t}_n}^\tau(v_\gamma^n)\), where \(\gamma_n\) is determined in the same way as \eqref{gamma_n}, by replacing \(\widehat{\Phi}_{\tilde{t}_n}^\tau\) with \(\psi_{\tilde{t}_n}^\tau\). Subsequently, we have 
\begin{equation}\label{RLRI1-v}
	u_\gamma^{n+1}:=\e^{i\gamma_n\tau\Delta}u_\gamma^n+\gamma_n\e^{i\tilde{t}_{n+1}\Delta}\psi_{\tilde{t}_n}^\tau(\e^{-i\tilde{t}_n\Delta}u_\gamma^n), \quad n\geq 0,
\end{equation}
where \(u_\gamma^{n}=\e^{i\tilde{t}_n\Delta}v_\gamma^{n}\). The scheme 
\eqref{RLRI1-v} shall be referred to as `RLRI1-v'.

The local error of the method \eqref{v_cubic_NLS} was analyzed in \cite{bruned_schratz_2022}, while the stability and global convergence analysis provided in \cite{ostermann2022second}. Apparently, the numerical flow  satisfies Assumptions \ref{Assumption} (i) and (ii) for any \(r > d/2\) with \(m= 2\), \(\beta = 2\).

\subsubsection{RLRI for general NLS \eqref{general_NLS}}

Following a procedure analogous to that described in \cite{ostermann2022second} for constructing the scheme above, we derive a second-order numerical scheme for solving \eqref{general_NLS}: \begin{equation}\label{u_general_NLS}
		u^{n+1}:=\Psi^\tau_{p,2}(u^n),\quad n\geq 0,\quad u^0=u_0,
	\end{equation}
	where
\begin{equation}\label{Psi_p}
		\begin{aligned}
			\Psi^\tau_{p,2}(f)=&\,\e^{i\tau\Delta}f-i\lambda\tau \e^{i\tau\Delta}\big[(f)^{p+1}\cdot \left( \varphi_1(-2i\tau\Delta)-\varphi_2(-2i\tau\Delta)\right)(\bar{f})^p \big] \\
			&\,-i\lambda\tau (\e^{i\tau\Delta}f)^{p+1}\cdot\big(\e^{i\tau\Delta}\varphi_2(-2i\tau\Delta)(\bar{f})^p \big)-\lambda^2\frac{\tau^2}{2}\e^{i\tau\Delta}\big[\abs{f}^{4p} f\big].
		\end{aligned}
	\end{equation}
This formulation corresponds to the scheme for the twisted variable  \(v:=\e^{-it\Delta}u\) as follows:
\begin{equation}\label{v_general_NLS}
		\begin{aligned}
			v^{n+1}:=v^n+\phi_{t_n}^\tau(v^n)
		\end{aligned}
\end{equation}
equipped with 
\begin{equation}\label{phi_t}
	\begin{aligned}
		\phi_{t}^\tau(f)=&-i\lambda\tau \e^{-it\Delta}\left[ (\e^{it\Delta}f)^{p+1}\cdot\left( \varphi_1(-2i\tau\Delta)-\varphi_2(-2i\tau\Delta)\right)(\e^{-it\Delta}\bar{f})^p \right] \\
		&-i\lambda\tau \e^{-i(t+\tau)\Delta}\left[ (\e^{i(t+\tau)\Delta}f)^{p+1}\cdot\varphi_2(-2i\tau\Delta)(\e^{-i(t-\tau)\Delta}\bar{f})^p \right]\\
		&-\lambda^2\frac{\tau^2}{2}\e^{-it\Delta}\left[ \abs{\e^{it\Delta}f}^{4p} (\e^{it\Delta}f)\right].
	\end{aligned}
\end{equation}
This scheme also satisfies Assumptions \ref{Assumption} (i) and (ii) for any \(r > d/2\) with \(m= 2\), \(\beta = 2\).
Applying the general framework outlined in Section \ref{general construction}, we propose the second-order scheme for the NLS \eqref{general_NLS} as follows: 
\begin{equation}\label{relax_v_NLS}
	v_\gamma^{n+1}:=v_\gamma^n+\gamma_n\phi_{\tilde{t}_n}^\tau(v_\gamma^n),\quad (n\geq 0),
\end{equation}
where \(\gamma_n\) can be determined in the same manner as in \eqref{gamma_n}.

{\begin{remark}
		While our focus is on the NLS equation, the above approach is readily applicable to a wide range of dispersive nonlinear PDEs and more invariants, which will be the subject of our future work.
\end{remark}}

\section{Numerical analysis} \label{CA-SP}
In this section, we rigorously show the conservation property and convergence analysis of the obtained RLRIs-v. Moreover, we will study the order reduction if we add the parameter to the LRIs of \(u\).
\subsection{Conservation property} We first study the structure-preserving property of the derived RLRIs-v.
\begin{mytheo}[Conservation property of general RLRIs-v]\label{Conservation}
	The RLRIs-v \eqref{RLRIs-v} exactly preserve the mass, namely,
	\[\norm{u_\gamma^n}_{L^2}=\norm{u_0}_{L^2},\quad n\geq 0.\]
\end{mytheo}
\begin{proof}
	In the case \(\widehat{\Phi}_{\tilde{t}_n}^\tau(v_\gamma
	^n)=0\), we obviously have \(v_\gamma^{n+1}=\tilde{v}^{n+1}\) and \(\norm{\tilde{v}^{n+1}}_{L^2}=\norm{v_\gamma^{n}}_{L^2}\). On the other hand, if \(\widehat{\Phi}_{\tilde{t}_n}^\tau(v_\gamma
	^n)\neq 0\), from the definition of the relaxation parameter \(\gamma_n\), it follows that \(\norm{v_\gamma^{n+1}}_{L^2}=\norm{v_\gamma^{n}}_{L^2}\). Twisting them back to \(u_\gamma^n\) and \(u_\gamma^{n+1}\) and using the identity \(\e^{it\Delta}\) in \(L^2\),   it automatically reaches the fact  \(\norm{u_\gamma^{n+1}}_{L^2}=\norm{u_\gamma^{n}}_{L^2}\). This indicates that \(\norm{u_\gamma^{n+1}}_{L^2}=\norm{u_0}_{L^2}\) for \(n\geq 0\), which gives the desired result.
\end{proof}

\subsection{General framework of convergence analysis}

For the convenience of error estimates, we further assume the following properties hold for all \(0<\tau\leq\tau_0\).
\begin{assumption}\label{Assumption2}
For \(\Phi_\gamma^n\) \eqref{relax_general_v_LRI}, \(v\in H^{r+\beta}\) with \(\ r\geq0,\,\beta\geq0,\,r+\beta\geq2\), assume 

	\textnormal{(i)} There exists a continuous function \(C_1:\ \mathbb{R}_{\geq 0}\times \mathbb{R}_{\geq 0}\rightarrow  \mathbb{R}_{\geq 0}\) such that  
	\begin{align*}
		\abs{\gamma_n-1}&\leq\tau^{m-1}C_1\Big( \norm{u_0}_{L^2}, \norm{v}_{L^\infty(0,T;H^{r+\beta})} \Big).
	\end{align*}
	
	\textnormal{(ii)} There exists a {continuous} function \(C_2:\ \mathbb{R}_{\geq 0}\times \mathbb{R}_{\geq 0}\rightarrow  \mathbb{R}_{\geq 0}\) such that the numerical flow \(\widehat{\Phi}_{t}^\tau\) fulfills
	\begin{equation*}
		\norm{\gamma_n\widehat{\Phi}_{\tilde{t}_n}^\tau(v_\gamma^n)-\widehat{\Phi}_{\tilde{t}_n}^{\gamma_n\tau}(v_\gamma^n)}_{H^r}\leq\tau^{m+1}C_2\Big( \norm{u_0}_{L^2},\norm{v}_{L^\infty(0,T;H^{r+\beta})} \Big).
	\end{equation*}
\end{assumption}

Based on Assumptions \ref{Assumption} and \ref{Assumption2}, we now enter the main bulk of convergence analysis for RLRIs-v. First of all, we prove the stability of our new obtained methods \eqref{relax_general_v_LRI}.
\begin{prop}[Stability estimate]\label{Stability}
	Suppose that \(\widehat{\Phi}_{t}^\tau\) satisfies Assumption \ref{Assumption} (i) and Assumption \ref{Assumption2} (i) for some \(r\geq 0\). Then there is a continuous function \(C\) such that the method \eqref{relax_general_v_LRI} satisfies
	\begin{equation}
		\norm{\Phi_\gamma^n(v(\tilde{t}_n))-\Phi_\gamma^n(v_\gamma^n)}_{H^r}\leq \Big(1+\tau C \big(\norm{v(\tilde{t}_n)}_{H^r},\norm{v_\gamma^n}_{H^r}\big)\Big)\norm{v(\tilde{t}_n)-v_\gamma^n}_{H^r}.
	\end{equation}
\end{prop}
\begin{proof}
	According to Assumption \ref{Assumption} (i), Assumption \ref{Assumption2} (i) and triangle inequality, we have 
	{\begin{align*}
			&\norm{\Phi_\gamma^n(v(\tilde{t}_n))-\Phi_\gamma^n(v_\gamma^n)}_{H^r}\\
			&\leq\norm{v(\tilde{t}_n)-v_\gamma^n}_{H^r}+\abs{\gamma_n}\norm{\widehat{\Phi}_{\tilde{t}_n}^\tau(v(\tilde{t}_n))-\widehat{\Phi}_{\tilde{t}_n}^\tau(v_\gamma^n)}_{H^r}\\
			&\leq \norm{v(\tilde{t}_n)-v_\gamma^n}_{H^r}+(1+C_1\tau^{m-1})\tau L\big(\norm{v(\tilde{t}_n)}_{H^r}, \norm{v_\gamma^n}_{H^r}\big) \norm{v(\tilde{t}_n)-v_\gamma^n}_{H^r},
	\end{align*}}where \(C_1:=C_1\left( \norm{u_0}_{L^2}, \norm{v}_{L^\infty(0,T;H^{r+\beta})} \right)\). 
	{Then for sufficiently small \(\tau\)}, it is easy to see
	\begin{align*}
		\norm{\Phi_\gamma^n(v(\tilde{t}_n))-\Phi_\gamma^n(v_\gamma^n)}_{H^r}\leq& \norm{v(\tilde{t}_n)-v_\gamma^n}_{H^r}\\
		&+2\tau L\big(\norm{v(\tilde{t}_n)}_{H^r}, \norm{v_\gamma^n}_{H^r}\big)\norm{v(\tilde{t}_n)-v_\gamma^n}_{H^r},
	\end{align*}which clearly implies the desired bound.
\end{proof}
For the sake of brevity, we define \[K_\delta=\norm{v}_{L^\infty(0,T;H^{\delta})},\quad \delta\geq 0,\]
and state the local error estimate as follows.
\begin{prop}[Local error]\label{Local_error}
	Under the conditions of Assumption \ref{Assumption} (ii) and Assumption \ref{Assumption2},  the local error is estimated as follows
	\begin{equation}
		\norm{v(\tilde{t}_{n+1})-\Phi_\gamma^n(v(\tilde{t}_n))}_{H^r}\leq \tau^{m+1}C\Big(\norm{u_0}_{L^2},\norm{v}_{L^\infty(0,T;H^{r+\beta})} \Big),\quad n\geq0.
	\end{equation}
\end{prop}
\begin{proof}
	Using Assumption \ref{Assumption} (ii), Assumption \ref{Assumption2} (i) and for sufficiently small \(\tau\), we have 
	{\begin{align*}
			\norm{\mathcal{R}^{\gamma_n\tau}_{\tilde{t}_n}}_{H^r}&=\norm{v(\tilde{t}_{n+1})-(v(\tilde{t}_n)+\widehat{\Phi}_{\tilde{t}_n}^{\gamma_n\tau}(v(\tilde{t}_{n})))}_{H^r}\\
			&\leq (\abs{\gamma_n}\tau)^{m+1}M\Big(\sup_{t\in[0,\gamma_n\tau]}\lVert v(\tilde{t}_n+t)\rVert_{H^{r+\beta}} \Big)\\
			&\leq ((\abs{\gamma_n-1}+1)\tau)^{m+1}M\Big(\sup_{t\in[0,\gamma_n\tau]}\lVert v(\tilde{t}_n+t)\rVert_{H^{r+\beta}} \Big)\\
			&\leq ((1+C_1\tau^{m-1})\tau)^{m+1}M\Big(\sup_{t\in[0,\gamma_n\tau]}\lVert v(\tilde{t}_n+t)\rVert_{H^{r+\beta}} \Big)\\
			&\leq 2\tau^{m+1}M\Big(\sup_{t\in[0,\gamma_n\tau]}\lVert v(\tilde{t}_n+t)\rVert_{H^{r+\beta}} \Big).
	\end{align*}}
	Then by triangle inequality and Assumption \ref{Assumption2} (ii), one deduces 
	\begin{align*}
		\norm{v(\tilde{t}_{n+1})-\Phi_\gamma^n(v(\tilde{t}_n))}_{H^r}&\leq  \norm{\mathcal{R}^{\gamma_n\tau}_{\tilde{t}_n}}_{H^r}+\norm{\gamma_n\widehat{\Phi}_{\tilde{t}_n}^\tau(v(\tilde{t}_n))-\widehat{\Phi}_{\tilde{t}_n}^{\gamma_n\tau}(v(\tilde{t}_n))}_{H^r}\\
		&\leq\tau^{m+1}C\Big(\norm{u_0}_{L^2},K_{r+\beta} \Big),
	\end{align*}
	which concludes the stated local error bound.
\end{proof}

To complete the convergence analysis, we now turn to the derivation of the global error estimate by combining the previously obtained local error bound with the stability result.

\begin{mytheo}[Global error]\label{global_error}
	Let \(u_\gamma^n\) be the numerical solution from the schemes RLRIs-v \eqref{RLRIs-v} for the NLS equation \eqref{general_NLS} up to some fixed time \(T > 0\). If \(\widehat{\Phi}_{t}^\tau\) satisfies Assumptions \ref{Assumption} and Assumption \ref{Assumption2} (ii), and \(\gamma_n\) satisfies Assumption \ref{Assumption2} (i), then for all \(0<\tau\leq\tau_0\) with \(\tau_0>0\) sufficiently small, we have
	\[\norm{u(\tilde{t}_n)-u_\gamma^n}_{H^r}\leq C\tau^m,\]
	where \(\tau_0,\,C\) only depend on \(\norm{u_0}_{L^2},\ \norm{u}_{L^\infty(0,T;H^{r+\beta})}\).
\end{mytheo}
{\begin{proof}
{As \(\e^{it\Delta}\) is a linear isometry on \(H^r\) for all \(t \in\mathbb{R}\), it suffices to show that 
\[\norm{v(\tilde{t}_n)-v_\gamma^n}_{H^r}\leq C\tau^m.\] We divide the proof into two parts.}

\textbf{(I) The verification of assumptions.}

The above convergence is based on Assumptions \ref{Assumption2}, which actually holds for a large class low regularity integrators. In this subsection, we present the rigorous proof of Assumptions \ref{Assumption2} to further demonstrate the effectiveness of RLRIs-v.

\noindent\textbf{\underline{Verification of Assumption \ref{Assumption2} (i).}}

For Assumption \ref{Assumption2} (i), from Theorem \ref{Conservation}, we know that the fact \(\norm{v_\gamma^n}_{L^2}=\norm{u_0}_{L^2}\neq
0\) holds for \(n\geq 0\). However, it needs to be proved that RLRIs-v are still of order-\(m\) under the same rough data as LRIs.
In order to certify this reality, we locally assume \(v_\gamma^n=v(\tilde{t}_n)\).

On the one hand, it follows that \(\tilde{v}^{n+1}:=v_\gamma^n+\Phi_{\tilde{t}_n}^\tau(v_\gamma^n)\) is an \((m+1)\)th-order approximation to \(v(\tilde{t}_n+\tau)\), which gives
\begin{align*}
	\big|\norm{\tilde{v}^{n+1}}^2_{L^2}-\norm{v_\gamma^n}^2_{L^2}\big|&=\left(\norm{\tilde{v}^{n+1}}_{L^2}+\norm{v_\gamma^n}_{L^2}\right)\big|\norm{\tilde{v}^{n+1}}_{L^2}-\norm{v_\gamma^n}_{L^2}\big|\\
	&\leq\left(\norm{\tilde{v}^{n+1}}_{L^2}+\norm{v_\gamma^n}_{L^2}\right)\norm{\tilde{v}^{n+1}-v(\tilde{t}_n+\tau)}_{L^2}\\
	&\leq\left(\norm{\tilde{v}^{n+1}}_{L^2}+\norm{v_\gamma^n}_{L^2}\right)C\tau^{m+1},
\end{align*}
where Assumption \ref{Assumption} (ii) is applied in the last inequality, and the constant \(C\) depends on \(\sup_{t \in [0, \tau]} \norm{v(\tilde{t}_n+t)}_{H^{r+\beta}}\).
Then for sufficiently small \(\tau\), by triangle inequality, we obtain
\[\norm{\tilde{v}^{n+1}}_{L^2}\leq\norm{v(\tilde{t}_n+\tau)}_{L^2}+\norm{\tilde{v}^{n+1}-v(\tilde{t}_n+\tau))}_{L^2}\leq \norm{u_0}_{L^2}+1.\]
Accordingly, one gets
\[\big|\norm{\tilde{v}^{n+1}}^2_{L^2}-\norm{v_\gamma^n}^2_{L^2}\big|\leq C\tau^{m+1},\]
where \(C\) depends on \(\norm{u_0}_{L^2}\) and \(\sup_{t \in [0, \tau]} \norm{v(\tilde{t}_n+t)}_{H^{r+\beta}}.\)

{On the other hand, by local assumption  \(v_\gamma^n=v(\tilde{t}_n)\), Assumption \ref{Assumption} (i) and Taylor expansion, we have
	\begin{align*}
		\widehat{\Phi}_{\tilde{t}_n}^\tau(v_\gamma^n)&=\tilde{v}^{n+1}-v_\gamma^n=v(\tilde{t}_n+\tau)-v(\tilde{t}_n)+\mathcal{R}_n^\tau\\
		&=\tau\partial_tv(\tilde{t}_n)+\int_{0}^{\tau}\partial_{tt}v(\tilde{t}_n+s)(\tau-s)\mathrm{d}s+\mathcal{R}_n^\tau,
	\end{align*}
	where \(\mathcal{R}_n^\tau=\tilde{v}^{n+1}-v(\tilde{t}_n+\tau)\).
	It can be easily computed that 
	\begin{align*}
		\partial_{tt}v(t)&=-\lambda\Delta \e^{-it\Delta}\big(\abs{\e^{it\Delta}v(t)}^{2p}\e^{it\Delta}v(t)\big)+\lambda (p+1)\e^{-it\Delta}\big(\abs{\e^{it\Delta}v(t)}^{2p}(\Delta \e^{it\Delta}v(t))\big)\\
		&\quad\, -\lambda p\e^{-it\Delta}\big(\abs{\e^{it\Delta}v(t)}^{2p-2}(\e^{it\Delta}v(t))^2(\Delta \e^{-it\Delta}\bar{v}(t))\big)\\
		&\quad\,-\lambda^2\e^{-it\Delta}\abs{\e^{it\Delta}v(t)}^{4p}\e^{it\Delta}v(t).
	\end{align*}
	Upon the above equations and Lemma \ref{bilinear_estimates}, we get 
	\[\big\|\widehat{\Phi}_{\tilde{t}_n}^\tau(v_\gamma^n)-\tau\partial_tv(\tilde{t}_n)\big\|_{L^2}\leq \tau^2C_*\big(\sup_{t \in [0, \tau]} \norm{v(\tilde{t}_n+t)}_{H^{r+\beta}}\big),\]
	where we use \(m\geq 2,\,r+\beta\geq 2\).
	Based on Cauchy-Schwarz inequality, for sufficiently small \(0<\tau<1\) satisfying \(\tau\big(2C_*\norm{\partial_tv(\tilde{t}_n)}_{L^2}+ C_*^2\big)\leq \frac{1}{2}\norm{\partial_tv(\tilde{t}_n)}^2_{L^2}\), one obtains
	\begin{align*}
		\abs{\mathcal{Q}_n}&:=\Big|2\operatorname{Re}\langle\widehat{\Phi}_{\tilde{t}_n}^\tau(v_\gamma^n)-\tau\partial_tv(\tilde{t}_n),\tau\partial_tv(\tilde{t}_n)\rangle+
		\norm{\widehat{\Phi}_{\tilde{t}_n}^\tau(v_\gamma^n)-\tau\partial_tv(\tilde{t}_n)}^2_{L^2}\Big|\\
		&\leq \tau^3\Big[2C_*\norm{\partial_tv(\tilde{t}_n)}_{L^2}+\tau C_*^2\Big]\leq \frac{1}{2}\tau^2\norm{\partial_tv(\tilde{t}_n)}^2_{L^2}.
	\end{align*}
    Furthermore, by observing the twisted equation \eqref{general_NLS_v} and applying the {embedding \(L^{2(2p+1)} \hookrightarrow L^2 \  (\text{for } p \geq 1)\)}, we are led to
	\begin{equation*}
		\begin{aligned}
			\norm{\partial_tv(\tilde{t}_n)}^2_{L^2}&=\lambda^2\norm{ (v(\tilde{t}_n))^{p+1}(\overline{v(\tilde{t}_n)})^p}^2_{L^2}\\
			&=\lambda^2\norm{v(\tilde{t}_n)}^{2(2p+1)}_{L^{2(2p+1)}}\geq\lambda^2 ((2\pi)^{-d})^{2p}\norm{v(\tilde{t}_n)}^{2(2p+1)}_{L^2}.
		\end{aligned}
	\end{equation*}
	Consequently, noticing that \(\norm{v(\tilde{t}_n)}_{L^2}=\norm{u_0}_{L^2}\), it is derived that
	\begin{align*}
		\norm{\widehat{\Phi}_{\tilde{t}_n}^\tau(v_\gamma^n)}^2_{L^2}&=\tau^2\norm{\partial_tv(\tilde{t}_n)}_{L^2}^2+\mathcal{Q}_n\\&\geq \frac{1}{2}\tau^2\norm{\partial_tv(\tilde{t}_n)}^2_{L^2}\geq \frac{1}{2}\tau^2\lambda^2 ((2\pi)^{-d})^{2p}\norm{u_0}^{2(2p+1)}_{L^2}.
	\end{align*}
	In the end,  the following bound occurs:
	\begin{equation}
		\abs{\gamma_n-1}\leq C\tau^{m-1},
	\end{equation}
	with \(C:=C\Big(\norm{u_0}_{L^2},\norm{v}_{L^\infty(0,T;H^{r+\beta})}\Big).\) }This confirms that Assumption \ref{Assumption2} (i) can be satisfied.

\noindent\textbf{\underline{Verification of Assumption \ref{Assumption2} (ii).}}

{The verification of Assumption \ref{Assumption2} (ii), which plays a key role in establishing the local error of RLRIs-v, must be carried out on a case-by-case basis. In this part, we present a detailed local error analysis for RLRI1-v as a representative case.}
In fact, we only need to prove Assumption \ref{Assumption2} (ii)  for RLRI1-v, and then local error result follows from Proposition \ref{Local_error}.

Letting \(\widehat{\Phi}_{\tilde{t}_n}=\psi_{\tilde{t}_n}^\tau\), which is defined as \eqref{psi_t}, we know 
\begin{align*}
	&\gamma_n\psi_{\tilde{t}_n}^\tau(f)-\psi_{\tilde{t}_n}^{\gamma_n\tau}(f)\\
	&=-i\gamma_n\lambda\tau \e^{-i\tilde{t}_n\Delta}\left[ (\e^{i\tilde{t}_n\Delta}f)^{2}\cdot\left(\varphi_1(-2i\tau\Delta)-\varphi_2(-2i\tau\Delta)\right)(\e^{-i\tilde{t}_n\Delta}\bar{f}) \right] \\
	&\quad-i\gamma_n\lambda\tau \e^{-i(\tilde{t}_n+\tau)\Delta}\left[ (\e^{i(\tilde{t}_n+\tau)\Delta}f)^{2}\cdot\varphi_2(-2i\tau\Delta)(\e^{-i(\tilde{t}_n-\tau)\Delta}\bar{f}) \right]\\
	&\quad+i\gamma_n\lambda\tau \e^{-i\tilde{t}_n\Delta}\left[ (\e^{i\tilde{t}_n\Delta}f)^{2}\cdot\left( \varphi_1(-2i\gamma_n\tau\Delta)-\varphi_2(-2i\gamma_n\tau\Delta)\right)(\e^{-i\tilde{t}_n\Delta}\bar{f}) \right] \\
	&\quad +i\gamma_n\lambda\tau \e^{-i(\tilde{t}_n+\gamma_n\tau)\Delta}\left[ (\e^{i(\tilde{t}_n+\gamma_n\tau)\Delta}f)^{2}\cdot\varphi_2(-2i\gamma_n\tau\Delta)(\e^{-i(\tilde{t}_n-\gamma_n\tau)\Delta}\bar{f}) \right]\\
	&\quad+\gamma_n(\gamma_n-1)\lambda^2\frac{\tau^2}{2}\e^{-i\tilde{t}_n\Delta}\left[ \abs{\e^{i\tilde{t}_n\Delta}f}^{4} (\e^{i\tilde{t}_n\Delta}f)\right]\\
	&:=\mathcal{\tilde{R}}_1^n(f)+\mathcal{\tilde{R}}_2^n(f)+\mathcal{\tilde{R}}_3^n(f),
\end{align*}
where \(\gamma_n\) and \(\tilde{t}_n\) are defined as Second \ref{general construction}, and
\begin{align*}
	\mathcal{\tilde{R}}_1^n(f)&=i\gamma_n\lambda\tau \e^{-i\tilde{t}_n\Delta}\Big[ (\e^{i\tilde{t}_n\Delta}f)^{2}\cdot\left( \varphi_1(-2i\gamma_n\tau\Delta)-\varphi_1(-2i\tau\Delta)\right)(\e^{-i\tilde{t}_n\Delta}\bar{f}) \\
	&\qquad\qquad\qquad\quad -(\e^{i\tilde{t}_n\Delta}f)^{2}\cdot\left(\varphi_2(-2i\gamma_n\tau\Delta)+\varphi_2(-2i\tau\Delta)\right)(\e^{-i\tilde{t}_n\Delta}\bar{f})\Big], \\
	\mathcal{\tilde{R}}_2^n(f)&=i\gamma_n\lambda\tau \e^{-i(\tilde{t}_n+\gamma_n\tau)\Delta}\left[ (\e^{i(\tilde{t}_n+\gamma_n\tau)\Delta}f)^{2}\cdot\varphi_2(-2i\gamma_n\tau\Delta)(\e^{-i(\tilde{t}_n-\gamma_n\tau)\Delta}\bar{f}) \right]\\
	&\quad-i\gamma_n\lambda\tau \e^{-i(\tilde{t}_n+\tau)\Delta}\left[ (\e^{i(\tilde{t}_n+\tau)\Delta}f)^{2}\cdot\varphi_2(-2i\tau\Delta)(\e^{-i(\tilde{t}_n-\tau)\Delta}\bar{f}) \right],\\
	\mathcal{\tilde{R}}_3^n(f)&=\gamma_n(\gamma_n-1)\lambda^2\frac{\tau^2}{2}\e^{-i\tilde{t}_n\Delta}\left[ \abs{\e^{i\tilde{t}_n\Delta}f}^{4} (\e^{i\tilde{t}_n\Delta}f)\right].
\end{align*}
Then we have 
\begin{align*}
	\mathcal{F}\Big(\tilde{\mathcal{R}}^n_1(f)\Big)(\boldsymbol{\xi})
	&=i\gamma_n\lambda\tau\int_{\boldsymbol{\xi}=\boldsymbol{\xi}_1+\boldsymbol{\xi}_2+\boldsymbol{\xi}_3}\e^{i\tilde{t}_n\alpha_3}\Big[\left( \varphi_1(2i\gamma_n\tau\abs{\boldsymbol{\xi}_3}^2)-\varphi_1(2i\tau\abs{\boldsymbol{\xi}_3}^2)\right)\\
	&\quad \qquad-\left( \varphi_2(2i\gamma_n\tau\abs{\boldsymbol{\xi}_3}^2)-\varphi_2(2i\tau\abs{\boldsymbol{\xi}_3}^2)\right)\Big]\hat{f}(\boldsymbol{\xi}_1)\hat{f}(\boldsymbol{\xi}_2)\hat{\bar{f}}(\boldsymbol{\xi}_3)(\mathrm{d}\boldsymbol{\xi}_1)(\mathrm{d}\boldsymbol{\xi}_2),
\end{align*}
where \(\alpha_3=\abs{\boldsymbol{\xi}}^2-\abs{\boldsymbol{\xi}_1}^2-\abs{\boldsymbol{\xi}_2}^2+\abs{\boldsymbol{\xi}_3}^2\). Using Lemma \ref{inequalities} yields 
\begin{align*}
	\Big|\mathcal{F}\Big(\tilde{\mathcal{R}}^n_1(f)\Big)(\boldsymbol{\xi})\Big|&\lesssim \abs{\lambda\gamma_n(\gamma_n-1)}\tau^2\int_{\boldsymbol{\xi}=\boldsymbol{\xi}_1+\boldsymbol{\xi}_2+\boldsymbol{\xi}_3}\abs{\boldsymbol{\xi}_3}^2\abs{\hat{f}(\boldsymbol{\xi}_1)\hat{f}(\boldsymbol{\xi}_2)\hat{\bar{f}}(\boldsymbol{\xi}_3)}(\mathrm{d}\boldsymbol{\xi}_1)(\mathrm{d}\boldsymbol{\xi}_2)\\
	&\lesssim \abs{\lambda\gamma_n(\gamma_n-1)}\tau^2 \mathcal{F}\Big(\tilde{f}^2\cdot(-\Delta \tilde{f})\Big)(\boldsymbol{\xi}),
\end{align*}
where \(\tilde{f}\) refers to the function with \(\mathcal{F}(\tilde{f})(\boldsymbol{\xi})=\abs{\hat{f}(\boldsymbol{\xi})}\).
Combining the last inequality with Lemma \ref{bilinear_estimates} gives
\[\norm{\tilde{\mathcal{R}}^n_1(f)}_{H^r}\lesssim \abs{\lambda\gamma_n(\gamma_n-1)}\tau^2\norm{\tilde{f}}_{H^{r+2}}^3\lesssim \abs{\lambda\gamma_n(\gamma_n-1)}\tau^2\norm{f}_{H^{r+2}}^3.\]
In a similar way, we obtain
\[\norm{\tilde{\mathcal{R}}^n_2(f)}_{H^r}\lesssim \abs{\lambda\gamma_n(\gamma_n-1)}\tau^2\norm{f}_{H^{r+2}}^3,\quad\norm{\tilde{\mathcal{R}}^n_3(f)}_{H^r}\lesssim \abs{\lambda^2\gamma_n(\gamma_n-1)}\tau^2\norm{f}_{H^{r+2}}^5.\]
Finally, it is obtained that
\[\norm{\gamma_n\psi_{\tilde{t}_n}^\tau(f)-\psi_{\tilde{t}_n}^{\gamma_n\tau}(f)}_{H^r}\leq \tau^3C_2\left( \norm{u_0}_{L^2},K_{r+2} \right),\]
which completes the proof of Assumption \ref{Assumption2} (ii), for RLRI1-v.

\textbf{(II) Error bounds.}

		Letting \(N\) be the total number of time steps, we consider the two cases \(\tilde{t}_N= T\) and \(T-\gamma_{N-1}\tau<\tilde{t}_{N-1}< T\), separately.
		Define the error function \(e_\gamma^k:=v(\tilde{t}_k)-v_\gamma^k,\, 0\leq k\leq N\).
		
		\textbf{Case 1: \(\tilde{t}_N= T\).} We shall carry out an induction proof on the boundedness of the numerical solution
		\[\norm{v_\gamma^n}_{H^r}\leq K_r+1,\quad 0\leq n\leq N,\]
		i.e., \(v_\gamma^n\in H^r.\) It obviously holds for \(n=0\) since we choose  \(v_\gamma^0=v(0)=u_0\). Suppose it is correct for \(0\leq n \leq N-1\), and we aim for the validity at \(N\).
		With the aid of the triangle inequality, we obtain for \( 0 \leq n \leq N - 1 \),
		\begin{equation}\label{general_local_error1}
			\norm{e_\gamma^{n+1}}_{H^r}\leq \norm{v(\tilde{t}_{n+1})-\Phi_\gamma^n(v(\tilde{t}_n))}_{H^r}+\norm{\Phi_\gamma^n(v(\tilde{t}_n))-\Phi_\gamma^n(v_\gamma^n)}_{H^r}.
		\end{equation}
		Thus, utilizing Propositions \ref{Stability} and \ref{Local_error}, it can be concluded that the regularity assumption \( v_\gamma^n \in H^r \) for all \( 0 \leq n \leq N-1 \) ensures the following result
		\begin{align*}
			\norm{e_\gamma^{n+1}}_{H^r}&\leq C\left(\norm{u_0}_{L^2},K_{r+\beta}\right)\tau^{m+1}+\left(1+\tau C\big(K_r,\norm{v_\gamma^n}_{H^r}\big)\right)\norm{e_\gamma^n}_{H^r}\\
			&\leq \tilde{C}\tau^{m+1}+(1+\widehat{C}\tau)\big(\tilde{C}\tau^{m+1}+(1+\widehat{C}\tau)\norm{e_\gamma^{n-1}}_{H^r}\big)\\
			&\leq \tilde{C}T\e^{\widehat{C}T}\tau^m,
		\end{align*}
		where \(\tilde{C}:=C\left(\norm{u_0}_{L^2},K_{r+\beta}\right)\),  \(\widehat{C}:=C\big(K_r,\norm{v_\gamma^n}_{H^r}\big)\) and the last inequality follows from the discrete Gronwall's lemma. We consequently obtain
		\[\norm{v_\gamma^N}_{H^r}\leq \norm{v(T)}_{H^r}+\norm{e_\gamma^N}_{H^r}\leq K_r+1,\]  and the induction is done which ends the proof.
		
		\textbf{Case 2: \(T-\gamma_{N-1}\tau<\tilde{t}_{N-1}< T\).} Similarly to Case 1, we also have \eqref{general_local_error1} for \(0\leq n\leq N-2\). For the last time step, let \(\tau_N=T-\tilde{t}_{N-1}\). Then it is obvious that \(\tau_N\leq 2\tau\) according to the estimation of \(\gamma_{N-1}\). Applying the original method satisfying Assumption \ref{Assumption} to approximate the endpoint \(T\) immediately gives  
		\begin{align*}
			\norm{e^N}_{H^r}:=&\,\norm{v(\tilde{t}_{n}+\tau_N)-\Phi^{\tau_N}_{\tilde{t}_n}(v_\gamma^n)}_{H^r}\quad (n=N-1)\\
			\leq&\, \norm{v(\tilde{t}_{n}+\tau_N)-\Phi^{\tau_N}_{\tilde{t}_n}(v(\tilde{t}_n))}_{H^r}+\norm{\Phi^{\tau_N}_{\tilde{t}_n}(v(\tilde{t}_n))-\Phi_{\tilde{t}_n}^{\tau_N}(v_\gamma^n)}_{H^r}\\
			\leq&\, M\left(K_{r+\beta}\right)\tau_N^{m+1}+\big(1+\tau_N L\,(\norm{v(\tilde{t}_n)}_{H^r}, \norm{v_\gamma^n}_{H^r})\big) \norm{e_\gamma^n}_{H^r}.
		\end{align*}
		Combining the above inequality with \eqref{general_local_error1} for \(0 \leq n \leq N-2\), and applying the discrete Gronwall's inequality under the induction assumption \(\norm{v_\gamma^n}_{H^r}\leq K_r+1\) for \(0 \leq n \leq N-1\), as established in Case~1, we obtain the desired error bound.
\end{proof}}

\subsection{Order reduction in the relaxation of \(u\)}\label{appen} 
It seems that using \(v\)-relaxation instead of directly relaxing \(u\) might be a roundabout approach.
However, if we add the relaxation parameter \(\gamma_n\) to the low regularity integrators for \(u\), then it may lead to order reduction under the same low regularity assumption as original scheme. 
In this part, we show the case where we face with order reduction to verify the rationality of \(v\)-relaxation, and we set the second order method \eqref{u_NLS} with \eqref{Psi_1} for cubic NLS \eqref{eq_u_cubic_NLS} with \(\lambda=1\) as an example to rigorously prove this fact.

The relaxation of \(u\) for scheme \eqref{u_NLS} with \(\lambda=1\), abbreviated as `RLRI-u', can be written as 
\begin{equation}\label{RLRI-u}
	\begin{aligned}
    u_\gamma^{n+1}:=&\,u_\gamma^n+\gamma_n\left(\Psi^\tau(u^n_\gamma)-u^n_\gamma\right),\quad n\geq0.
	\end{aligned}
\end{equation}
Here and after, \(\Psi^\tau\) is abbreviation of \(\Psi^\tau_{1,2}\) in \eqref{Psi_1}.

Then we give the function: \[S(\gamma)=\int_{\mathbb{T}^d}\big|u_\gamma^n+\gamma\left(\Psi^\tau(u^n_\gamma)-u^n_\gamma\right)\big|^2-\abs{u_\gamma^n}^2\mathrm{d}\mathbf{x},\] whose non-zero root is the desired \(\gamma_n\) for \eqref{RLRI-u}. Following the same steps as Subsection  \ref{general construction}, one can respectively prove the well-posedness of scheme \eqref{RLRI-u} as
\begin{equation}\label{gamma_n_u}
	\gamma_n=\begin{cases}
		1-\frac{\norm{\Psi^\tau(u_\gamma^n)}^2_{L^2}-\norm{u_\gamma^n}^2_{L^2}}{\norm{\Psi^\tau(u_\gamma^n)-u_\gamma^n}^2_{L^2}},\quad & \norm{\Psi^\tau(u_\gamma^n)-u_\gamma^n}_{L^2}\neq0,\\
		1,\quad &\norm{\Psi^\tau(u_\gamma^n)-u_\gamma^n}_{L^2}=0,
	\end{cases}
\end{equation}
and derive the estimation of \(\gamma_n\):
\begin{equation}\label{estimate_gamma_u}
    \abs{\gamma_n-1}\leq C\tau,
\end{equation}
where \(C\) depends on \({\norm{u_0}_{L^2}},\,\norm{u}_{L^\infty(0,T;H^{r+2})}.\) Here we omit it for brevity. Now we are in the position to prove the convergence of the new scheme \eqref{RLRI-u}.

To begin with, letting \(\tau_n:=\gamma_n\tau\) as the \(n\)th time step size and using scheme \eqref{u_NLS} to approximate \(u(\tilde{t}_n + \gamma_n\tau)\) as
\begin{equation}\label{exact_u_NLS}
	u(\tilde{t}_n + \gamma_n\tau) = \Psi^{\gamma_n\tau}(u(\tilde{t}_n))+\mathcal{E}^{\gamma_n\tau}_{\tilde{t}_n}, 
\end{equation}
where \(\norm{\mathcal{E}^{\gamma_n\tau}_{\tilde{t}_n}}_{H^r}\leq C\tau^3\), and \(C\) depends on \(\sup_{t\in[0,\gamma_n\tau]}\lVert u(\tilde{t}_n+t)\rVert_{H^{r+2}}.\)
By inserting the exact solution into the scheme \eqref{RLRI-u}, one gets
\begin{equation}\label{exact_u_NLS_into_numerical}
	u(\tilde{t}_n + \gamma_n\tau) =u(\tilde{t}_n)+\gamma_n\left(\Psi^\tau(u(\tilde{t}_n))-u(\tilde{t}_n)\right)+ \tilde{\mathcal{E}}^n(u(\tilde{t}_n)).
\end{equation}
Subtracting  \eqref{exact_u_NLS_into_numerical} from \eqref{exact_u_NLS}, we obtain 
\begin{equation*}
	\begin{aligned}
		\tilde{\mathcal{E}}^n(u(\tilde{t}_n))=&\, \Psi_{\gamma_n\tau}(u(\tilde{t}_n))+\mathcal{E}^{\gamma_n\tau}_{\tilde{t}_n}-u(\tilde{t}_n)-\gamma_n\left(\Psi^\tau(u(\tilde{t}_n))-u(\tilde{t}_n)\right)\\
		=&\,(\e^{i \gamma_n\tau \Delta}-1)u(\tilde{t}_n) -\gamma_n (\e^{i\tau\Delta}-1)u(\tilde{t}_n)
		+\mathcal{E}^{\gamma_n\tau}_{\tilde{t}_n}\\
		&\,-i\gamma_n\tau \e^{i\gamma_n\tau\Delta}\big[ (u(\tilde{t}_n))^2 \cdot\left( \varphi_1(-2i\gamma_n\tau\Delta)-\varphi_2(-2i\gamma_n\tau\Delta)\right)\overline{u(\tilde{t}_n)}\big] \\
		&\,-i\gamma_n\tau (\e^{i\gamma_n\tau\Delta}u(\tilde{t}_n))^2\cdot\big(\e^{i\gamma_n\tau\Delta}\varphi_2(-2i\gamma_n\tau\Delta)\overline{u(\tilde{t}_n)} \big)\\
		&\,-\frac{(\gamma_n\tau)^2}{2}\e^{i\gamma_n\tau\Delta}\big[\abs{u(\tilde{t}_n)}^4 u(\tilde{t}_n)\big]\\
		&\,+ i\gamma_n\tau \e^{i\tau\Delta}\big[ (u(\tilde{t}_n))^2\cdot( \varphi_1(-2i\tau\Delta)-\varphi_2(-2i\tau\Delta))\overline{u(\tilde{t}_n)} \big]\\
		&\,+\gamma_n\Big[i\tau (\e^{i\tau\Delta}u(\tilde{t}_n))^2\cdot\big(\e^{i\tau\Delta}\varphi_2(-2i\tau\Delta)\overline{u(\tilde{t}_n)} \big) +\frac{\tau^2}{2}\e^{i\tau\Delta}\big[\abs{u(\tilde{t}_n)}^4 u(\tilde{t}_n)\big]\Big]\\:=&\,\tilde{\mathcal{E}}^n_1(u(\tilde{t}_n))+\tilde{\mathcal{E}}^n_2(u(\tilde{t}_n))+\tilde{\mathcal{E}}^n_3(u(\tilde{t}_n))+\tilde{\mathcal{E}}^n_4(u(\tilde{t}_n))+\mathcal{E}^{\gamma_n\tau}_{\tilde{t}_n},
	\end{aligned}
\end{equation*}
where \begin{align*}
	\tilde{\mathcal{E}}^n_1(f)&=\big[(\e^{i \gamma_n\tau \Delta}-1)-\gamma_n (\e^{i\tau\Delta}-1)\big]f,\\
	\tilde{\mathcal{E}}^n_2(f)&=i\gamma_n\tau\Big[ \e^{i\tau\Delta}\big[  f^2\cdot\left( \varphi_1(-2i\tau\Delta)-\varphi_2(-2i\tau\Delta)\right)\bar{f} \big]\\
	&\quad- \e^{i\gamma_n\tau\Delta}\big[ f^2\cdot\left( \varphi_1(-2i\gamma_n\tau\Delta)-\varphi_2(-2i\gamma_n\tau\Delta)\right)\bar{f} \big]\Big] \\
	&=i\gamma_n\tau (\e^{i\tau\Delta}-\e^{i\gamma_n\tau\Delta})\big[  f^2\cdot\left( \varphi_1(-2i\tau\Delta)-\varphi_2(-2i\tau\Delta)\right)\bar{f} \big]\\
	&\quad+i\gamma_n\tau \e^{i\gamma_n\tau\Delta}\big[ f^2\cdot( \varphi_1(-2i\tau\Delta)-\varphi_2(-2i\tau\Delta)- \varphi_1(-2i\gamma_n\tau\Delta)\\
	&\qquad\qquad\qquad\qquad\qquad\qquad\qquad\qquad\qquad\quad\  \ +\varphi_2(-2i\gamma_n\tau\Delta))\bar{f} \big]\\
	&:=\tilde{\mathcal{E}}^n_{2,1}(f)+\tilde{\mathcal{E}}^n_{2,2}(f),\\
	\tilde{\mathcal{E}}^n_3(f)&=i\gamma_n\tau\Big[ (\e^{i\tau\Delta}f)^2\cdot\big(\e^{i\tau\Delta}\varphi_2(-2i\tau\Delta)\bar{f} \big)- (\e^{i\gamma_n\tau\Delta}f)^2\cdot\big(\e^{i\gamma_n\tau\Delta}\varphi_2(-2i\gamma_n\tau\Delta)\bar{f} \big)\Big],\\
	\tilde{\mathcal{E}}^n_4(f)&=\gamma_n\frac{\tau^2}{2}\Big[\e^{i\tau\Delta}\big[\abs{f}^4 f\big]-\gamma_n\e^{i\gamma_n\tau\Delta}\big[\abs{f}^4 f\big]\Big]\\
	&=\gamma_n\frac{\tau^2}{2}\Big[(\e^{i\tau\Delta}-\e^{i\gamma_n\tau\Delta}-(\gamma_n-1)\e^{i\gamma_n\tau\Delta})\big[\abs{f}^4 f\big]\Big].
\end{align*}
Using \(\abs{\e^{ix}-1}\leq \abs{x}\) and \(\abs{\e^{ix}-\e^{iy}}\leq\abs{x-y}\) for any \(x,y\in\mathbb{R}\) yields
\begin{align*}
\langle\boldsymbol{\xi}\rangle^r\Big|\mathcal{F}\Big(\tilde{\mathcal{E}}^n_1(f)\Big)(\boldsymbol{\xi})\Big|&=\langle\boldsymbol{\xi}\rangle^r\big|\big[(\e^{-i \gamma_n\tau \abs{\boldsymbol{\xi}}^2}-1)-\gamma_n (\e^{-i\tau \abs{\boldsymbol{\xi}}^2}-1)\big]\hat{f}(\boldsymbol{\xi})\big|\\
	&=\langle\boldsymbol{\xi}\rangle^r\abs{(\e^{-i \gamma_n\tau \abs{\boldsymbol{\xi}}^2}-1)(1-\gamma_n)+\gamma_n (\e^{-i \gamma_n\tau \abs{\boldsymbol{\xi}}^2}-\e^{-i\tau \abs{\boldsymbol{\xi}}^2})}\abs{\hat{f}(\boldsymbol{\xi})}\\
	&\lesssim \abs{\gamma_n-1}\abs{\gamma_n}\tau\abs{\boldsymbol{\xi}}^{r+2}\abs{\hat{f}(\boldsymbol{\xi})}\lesssim \abs{\gamma_n-1}\abs{\gamma_n}\tau \Big|\mathcal{F}\Big((-\Delta)^{r/2+1}f\Big)(\boldsymbol{\xi})\Big|.
\end{align*}
Then by Plancherel’s identity, we obtain for any \(r>d/2\)
\[\norm{\tilde{\mathcal{E}}^n_1(f)}_{H^r}\lesssim\abs{\gamma_n-1}\abs{\gamma_n}\tau\norm{f}_{H^{r+2}}.\]
Following Lemma \ref{bilinear_estimates} and Lemma \ref{inequalities}, it holds that
\[\norm{\tilde{\mathcal{E}}^n_{2,1}(f)}_{H^r}\lesssim\abs{\gamma_n-1}\abs{\gamma_n}\tau^2\norm{f}^3_{H^{r+2}},\quad \norm{\tilde{\mathcal{E}}^n_{2,2}(f)}_{H^r}\lesssim\abs{\gamma_n-1}\abs{\gamma_n}\tau^2\norm{f}^3_{H^{r+2}},\]
and then we have \[\norm{\tilde{\mathcal{E}}^n_2(f)}_{H^r}\lesssim\abs{\gamma_n-1}\abs{\gamma_n}\tau^2\norm{f}^3_{H^{r+2}}.\]

In a similar way, we obtain that
\[\norm{\tilde{\mathcal{E}}^n_3(f)}_{H^r}\lesssim\abs{\gamma_n-1}\abs{\gamma_n}\tau^2\norm{f}^3_{H^{r+2}},\quad \norm{\tilde{\mathcal{E}}^n_4(f)}_{H^r}\lesssim\abs{\gamma_n-1}\abs{\gamma_n}\tau^2\norm{f}^5_{H^{r+2}}.\]
Combining the above results with estimate \eqref{estimate_gamma_u} concludes 
\[\norm{\tilde{\mathcal{E}}^n(f)}_{H^r}\leq\tau^2 C(\norm{u_0}_{L^2},\norm{f}_{L^\infty(0,T;H^{r+2})}),\] 
namely, the RLRI-u \eqref{RLRI-u} is first order for initial value \(u_0\in H^{r+2},\ r>d/2\), which shows the order reduction in this case.

{\begin{remark}
		Analogously, if we employ~\eqref{RLRI-u}, where $\gamma_n$ is determined by solving 
		\(\mathcal{H}(u_\gamma^{n+1})=\mathcal{H}(u_\gamma^n),
		\)
		to construct low-regularity energy-preserving methods, an order reduction still occurs. Addressing this order reduction, however, lies beyond the scope of the present paper and is left for future investigation.
		\end{remark}}
We have presented a rigorous and comprehensive numerical analysis of the newly proposed methods RLRIs-v. In the next section, numerical experiments will be conducted to validate the above results.

\section{Numerical experiments}\label{ne}
This section is dedicated to examining the numerical behavior of new schemes RLRIs-v \eqref{RLRIs-v}. In all of the experiments, we fix \(d=1\) and our spatial discretisation is a spectral method with \(K=2^{12}\) (unless otherwise stated) Fourier modes. Our initial conditions are of following two types: 

\noindent1. \emph{Smooth initial conditions.} \(u_0\in C^\infty(\mathbb{T})\) is of the form
\begin{equation}\label{smooth_initial_data}
	u_0(x)= \cos(x)/(2+\sin(x));
\end{equation}

\noindent2. \emph{Low-regularity initial data.} \(u_0\in H^\theta\ (\theta>1/2)\) is defined as 

\begin{equation}\label{rough_data}
	u_0=\abs{\partial_{x,K}}^{-\theta}\mathcal{U}^N,\quad \big(\abs{\partial_{x,K}}^{-\theta}\big)_k:=\begin{cases}
		\abs{k}^{-\theta}\ &\text{
			if} \ \ k\neq0,\\
		0\quad &\text{
			if} \ \ k=0,
	\end{cases}
\end{equation}
where \(\mathcal{U}^K=\operatorname{rand}(K,1)+i\operatorname{rand}(K,1)\in \mathbb{C}^K,\) and \(\operatorname{rand}(K,1)\) returns \(K\) uniformly distributed random numbers between \(0\) and \(1\). 

Both choices of initial data are normalized in \(L^2\) as \(u_0 \mapsto u_0/\|u_0\|_{L^2}\).
In particular, the practical performance of the aforementioned methods in terms of
\begin{itemize}
	\item convergence under rough data;
	\item high computational efficiency;
	\item precise long-time conservation of the \(L^2\)-norm, even up to near machine precision,
\end{itemize}
will be tested to verify our theoretical results.

In practical numerical simulation, we choose RLRI1-v \eqref{RLRI1-v} for solving one dimensional cubic nonlinear Schr\"{o}dinger equations of type \eqref{eq_u_cubic_NLS} (i.e., \(p=1,\,d=1,\,\Delta=\partial_x^2\) in \eqref{general_NLS}) as an example of RLRIs-v. 
In addition, to demonstrate the broad applicability of RLRIs-v \eqref{RLRIs-v}, we apply this technique to low regularity integrator formulated in \cite{Kn19}:
\(v^{n+1}:=\varphi_{t_n}^\tau(v^n)\),
where 
$$	\varphi_{t}^\tau(f)=\e^{-it\partial_x^2}\Big(\e^{i\lambda\tau\abs{\e^{it\partial_x^2f}}^2}\
e^{it\partial_x^2}f-i\lambda\big(J_1(\e^{it\partial_x^2}f)+J_2(\e^{it\partial_x^2}f)\big)\Big)
$$
equipped with 
\begin{align*}
	J_1(g)=&\frac{i}{2}\Big[\e^{-i\tau\partial_x^2}\partial_x^{-1}\big((\e^{-i\tau\partial_x^2}\partial_x^{-1}\bar{g})(\e^{i\tau
		\partial_x^2}v^2)\big)-\partial_x^{-1}((\partial_x^{-1}\bar{g})g^2)\Big]\\
	&\,+\tau\hat{\bar{g}}_0g^2+\tau\widehat{(\abs{g}^2g)}_0-\tau\hat{\bar{g}}_0\widehat{(g^2)}_0,\\
	J_2(g)=&\frac{i}{2}\Big[\e^{-i\tau\partial_x^2}(\e^{i\tau\partial_x^2}\partial_x^{-1}g)^2-(\partial_x^{-1}v)^2\Big]\bar{g}+\tau\hat{g}_0(2g-\hat{g}_0)\bar{g},
\end{align*}
which satisfies Assumption \ref{Assumption} for any \(r > 1/2\) and \(m= 2\), \(\beta = 2\).
In the same construction steps as \eqref{RLRIs-v} by substituting \(\varphi_{\tilde{t}_n}^\tau-I\) for \(\widehat{\Phi}_{\tilde{t}_n}^\tau\) in \eqref{relax_general_v_LRI}--\eqref{RLRIs-v}, we obtain RLRI2-v as 
\begin{equation}\label{RLRI2-v}
	u_\gamma^{n+1}:=\e^{i\gamma_n\tau\partial_x^2}u_\gamma^n+\gamma_n\e^{i\tilde{t}_{n+1}\partial_x^2}\big(\varphi_{t_n}^\tau(\e^{-i\tilde{t}_n\partial_x^2}u_\gamma^n)-\e^{-i\tilde{t}_n\partial_x^2}u_\gamma^n\big), \quad n\geq 0.
\end{equation}
Meanwhile, we also provide numerical results for RLRI-u \eqref{RLRI-u} with \eqref{gamma_n_u} as an illustrative example of the order reduction that arises when the relaxation technique is applied to LRIs for \(u\) under rough data.

Furthermore, for the purpose of comparison, we consider four well-performing second order numerical schemes for the cubic NLS equation from the literature:

\noindent-- Strang splitting \cite{Lubich2008} referred to as `Strang' is in the following form:
\begin{equation*}
	\begin{aligned}
		u^{n+1/2}_-&=\e^{i\frac{\tau}{2}\partial_x^2}u^n_S,\ \
		u^{n+1/2}_+=\e^{-i\lambda\tau\abs{u^{n+1/2}_-}^2}u^{n+1/2}_-,\ \
		u^{n+1}_S=\e^{i\frac{\tau}{2}\partial_x^2}u^{n+1/2}_+.
	\end{aligned}
\end{equation*}

\noindent-- The \(L^2\)-norm preserving Lawson method introduced in \cite{Celledoni2008} (see Example 3.2), which we refer to as `Lawson', is given by
\begin{equation*}
		\begin{aligned}
			L^n&=-i\lambda\Big|\e^{i\frac{\tau}{2}\partial_x^2}u^n_L+\frac{\tau}{2}L^n\Big|^2\left(\e^{i\frac{\tau}{2}\partial_x^2}u^n_L+\frac{\tau}{2}L^n\right),\ \
			u^{n+1}_L=\e^{i\tau\partial_x^2}u^n_L+ \tau \e^{i\frac{\tau}{2}\partial_x^2}L^n.
		\end{aligned}
\end{equation*}

\noindent-- Symplectic low regularity integrator (called by `SLRI') $u^{n+1}:=\Psi^\tau_{Sym} (u^n)$ developed in \cite{Maierhofer2022} (see Example 3.17) reads as:
\begin{align*}
	\Psi^\tau_{Sym}(f)
	&=\e^{i\tau\partial_x^2} f
	- i\lambda \left[ \frac{i}{2} \partial_x^{-1} \left( \left( \e^{-i\tau \partial_x^2} \overline{\partial_x^{-1} g} \right)\left( \e^{i\tau \partial_x^2}  g^2 \right) \right) - \frac{i}{2} \e^{i\tau\partial_x^2} \partial_x^{-1} \left(\overline{ \partial_x^{-1} g} g^2 \right) \right] \\
	&\quad - i\lambda \e^{i\tau\partial_x^2} \left[ \frac{i}{2} \bar{g}\e^{-i\tau \partial_x^2} \left( \e^{i\tau \partial_x^2} \partial_x^{-1} g\right)^2 
	- \frac{i}{2}\bar{g} \left( \partial_x^{-1} g \right)^2 - \tau \abs{g}^2g \right] \\
	&\quad - i\lambda \tau \left[\frac{1}{2\pi} \int_{\mathbb{T}} \abs{g}^2g- \bar{\hat{g}}_0g^2\, \mathrm{d}x 
	+ \bar{\hat{g}}_0\e^{i\tau\partial_x^2} g^2+ 2 \hat{g}_0 \e^{i\tau\partial_x^2} \abs{g}^2- \left( \hat{g}_0 \right)^2 \e^{i\tau\partial_x^2} \bar{g} \right],
\end{align*}
with \(g=(f+\e^{-i\tau\partial_x^2}\Psi^\tau_{Sym}(f))/2.\)
{This scheme achieves second-order convergence in the \(H^1\) error for \(u_0 \in H^3\) and preserves the \(L^2\)-norm.}

\noindent-- Low regularity integrator \eqref{u_NLS} \cite{bruned_schratz_2022,ostermann2022second} without relaxation, denoted by `LRI1' later.

All numerical experiments were performed using MATLAB R2024a on a standard laptop computer. Since MATLAB is based on numerical computations and subject to rounding errors, the direct use of \(\norm{u_\gamma^n}_{L^2}=\norm{v_\gamma^n}_{L^2}\) may introduce additional errors which, although small, are not negligible. To prevent the accumulation of such errors in computation of \(L^2\)-norm, we further refined the computation of \(\gamma_n\) when $\norm{\widehat{\Phi}_{\tilde{t}_n}^\tau(v_\gamma^n)}_{L^2}\neq0$ in \eqref{gamma_n} as  
\begin{equation}\label{refined_gamma_n}
	\gamma_n= 
		1-\Big(\norm{\Phi_{\tilde{t}_n}^\tau(v_\gamma^n)}^2_{L^2}-\norm{u_0}^2_{L^2}\Big)\Big/\norm{\widehat{\Phi}_{\tilde{t}_n}^\tau(v_\gamma^n)}^2_{L^2},\quad \norm{\widehat{\Phi}_{\tilde{t}_n}^\tau(v_\gamma^n)}_{L^2}\neq0,
\end{equation}to better preserve the \(L^2\)-norm in numerical experiments. As a matter of fact, \eqref{refined_gamma_n} is equivalent to \eqref{gamma_n} and it can be proved by recursion.  Firstly, this result holds for \(n=0\) since \(v_\gamma^0=u_0\). Then using \(\gamma_0\) we obtain \(\norm{v_\gamma^1}_{L^2}=\norm{u_0}_{L^2}\) so that it also holds for \(n=1\). By repeatedly applying this process, we can arrive at the conclusion that \eqref{refined_gamma_n} holds for all \(n\geq 0\). In other words, this result also follows directly from the Theorem \ref{Conservation}. Naturally,  \eqref{gamma_n_u} can be treated in the same fashion by replacing \(\norm{u_\gamma^n}_{L^2}\) with \(\norm{u_0}_{L^2}\) in practical tests. Besides, in the simulations, to testify the fact \(\abs{\gamma_n-1}\lesssim \tau^{m-1}\) (\(m=2\) in the involved methods), we use \(d(\gamma_n)\) to denote the average distances between \(\gamma_n\) and \(1\), namely, \(d(\gamma_n) :=\frac{1}{N+1}\sum_{n=0}^{N} \abs{\gamma_n-1},\) which are used to verify the estimation of the relaxation coefficient \(\gamma_n\).

First of all, we investigate the convergence behavior of our proposed schemes compared with existing methods. To this end, we consider initial data of various levels of regularity \(u_0\in H^2, \,u_0\in H^3\) with \(\theta= 2, 3\) in \eqref{rough_data} and \(u_0\in C^\infty \) in \eqref{smooth_initial_data} and measure the \(H^1 \) error at time \(T = 1\) for different time steps \(\tau\). Our reference solutions were computed with \(K = 2^{12}\) Fourier modes and a time step \(\tau = 5\times 10^{-5}\) using the second order method in \eqref{u_NLS} from \cite{ostermann2022second}. The results are shown in Figure \ref{fig:rough_data_error}. We clearly observe that the new methods RLRIs-v maintain the original convergence rates of order two at those levels of regularity as per Theorem \ref{global_error}. Particularly, the error curves of RLRI1-v and LRI1 are nearly identical, suggesting that applying relaxation to LRIs of \(v\) does not adversely affect the error performance of LRIs. {Furthermore, it should be pointed out that the error constant of RLRI1-v is smaller than that of RLRI2-v, which is just due to differences between the baseline methods.} However, RLRI-u shows a clear reduction in convergence order and an increase in numerical error relative to LRI1 under rough data. {These observations confirm our convergence analysis in Section \ref{CA-SP}.}
\begin{figure}[htbp]
	\centering
	
	\subfigure[\(u_0\in H^2\), \(\theta=2\) in 
	\eqref{rough_data} ]{
		\includegraphics[width=0.32\textwidth, height=3.9cm]{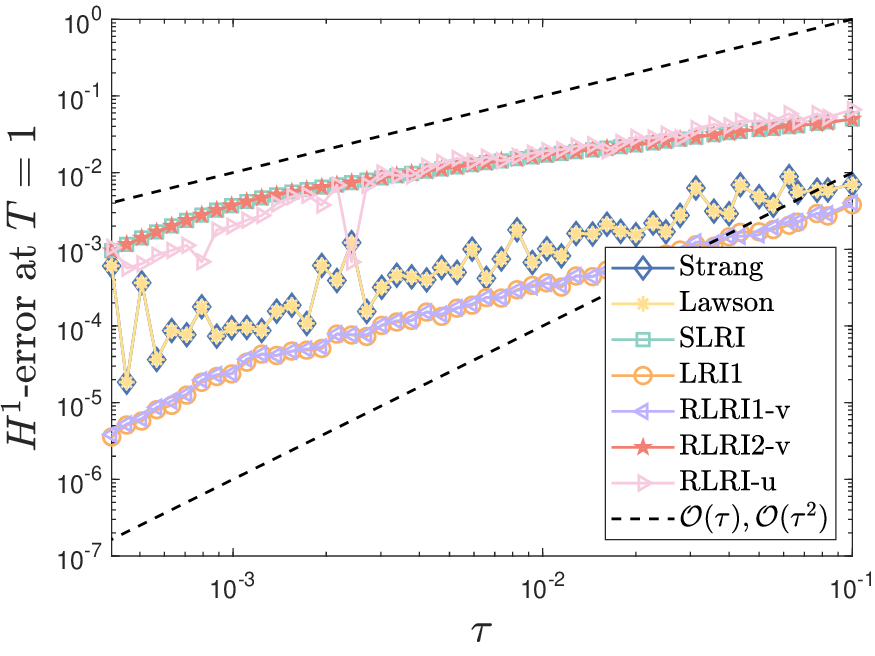}
	}
	\hspace{-1em}
	\subfigure[\(u_0\in H^3\), \(\theta=3\) in 
	\eqref{rough_data}  ]{
		\includegraphics[width=0.32\textwidth, height=3.9cm]{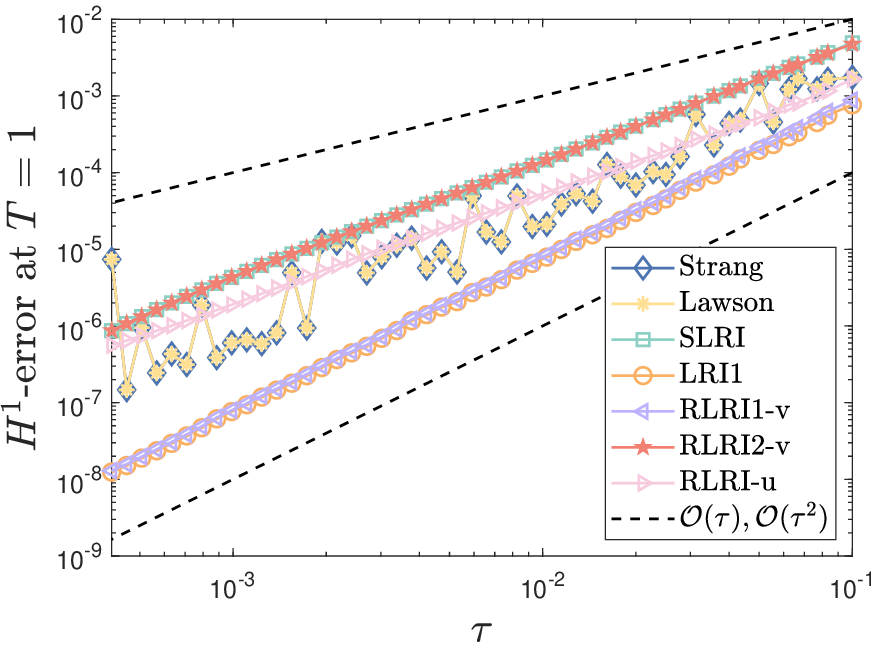}
	}
	\hspace{-1em}
	\subfigure[\(u_0\in C^\infty\) in 
	\eqref{smooth_initial_data}  ]{
		\includegraphics[width=0.32\textwidth, height=3.9cm]{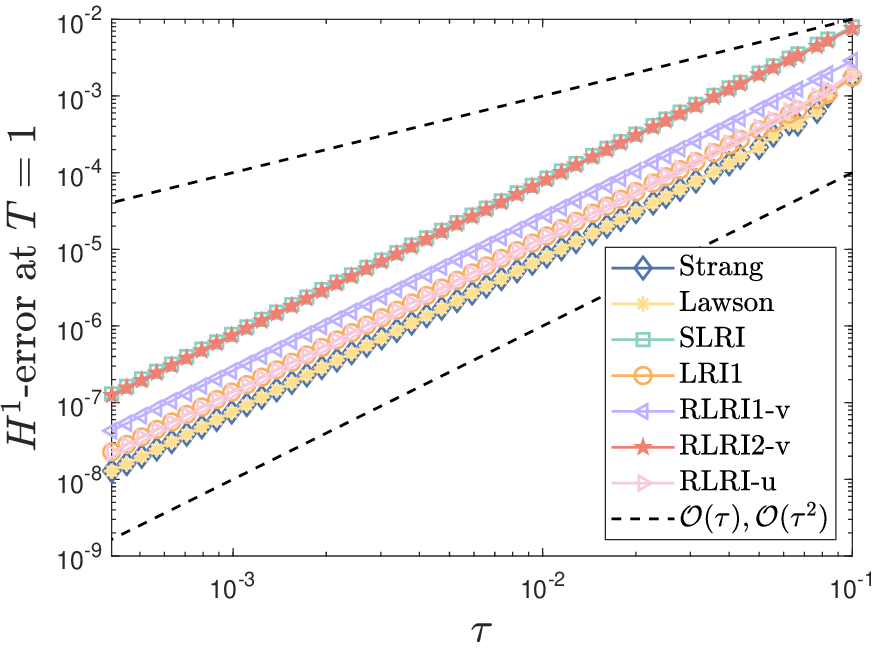}
	}
	
	\caption{\(H^1\)-error at \(T=1\) as a function of \(\tau\) for low-regularity and smooth initial data.}
	\label{fig:rough_data_error}
\end{figure}
\begin{figure}[htbp]
	\centering
	
	\subfigure[\(u_0\in H^2\), \(\theta=2\) in 
	\eqref{rough_data} ]{
		\includegraphics[width=0.32\textwidth, height=3.9cm]{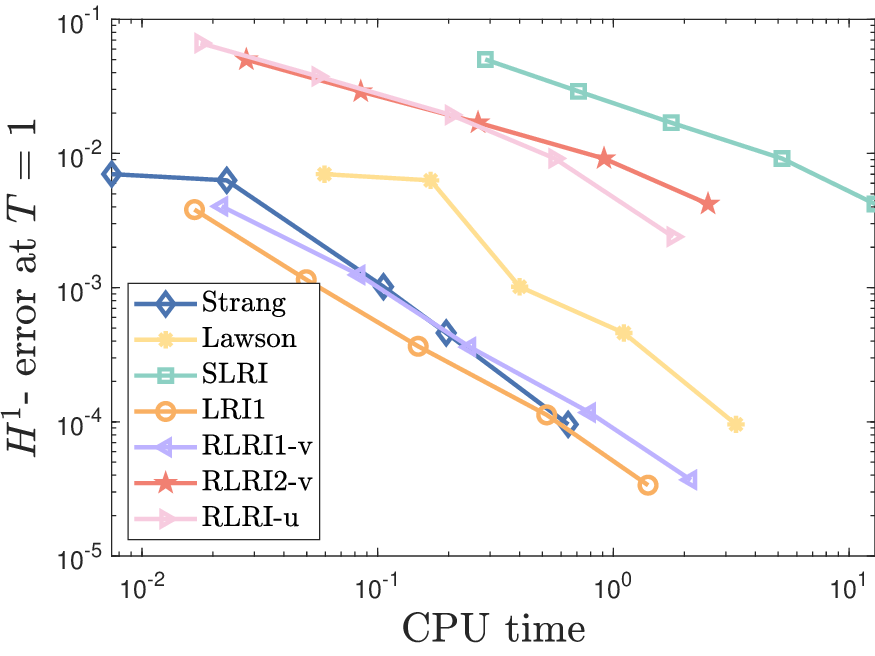}
	}
	\hspace{-1em}
	\subfigure[\(u_0\in H^3\), \(\theta=3\) in 
	\eqref{rough_data}  ]{
		\includegraphics[width=0.32\textwidth, height=3.9cm]{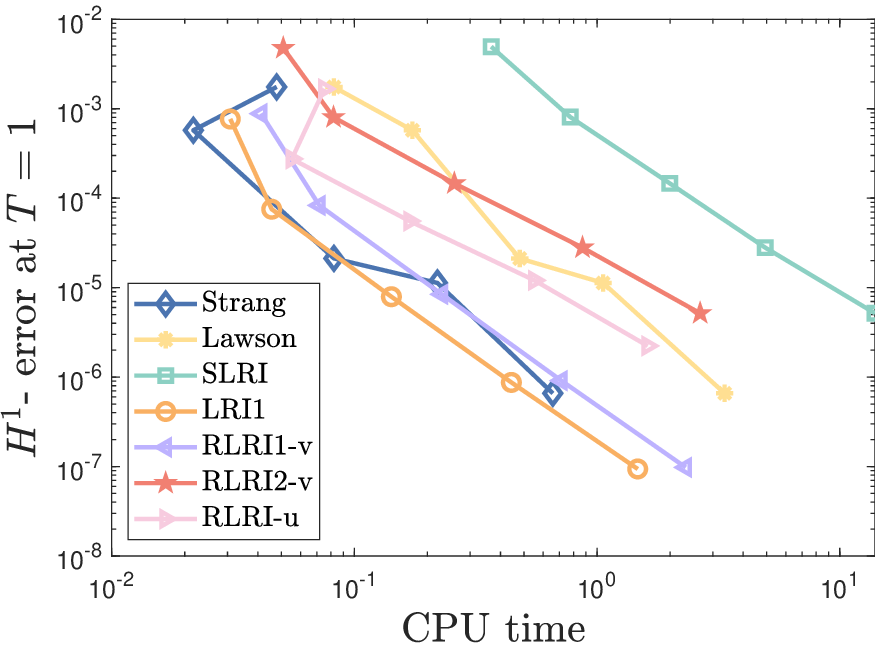}
	}
		\hspace{-1em}
	\subfigure[\(u_0 \in C^\infty\)  in 
\eqref{smooth_initial_data}  ]{
		\includegraphics[width=0.32\textwidth, height=3.9cm]{H3_cpu.eps}
	}
	
	\caption{\(H^1\)-error at \(T=1\) versus CPU time for low-regularity and smooth initial data.}
	\label{fig:rough_data_cpu}
\end{figure}

Thereafter, to show the high efficiency of the proposed methods, we compute the \(H^1\)-errors at time \(T= 1\) of these methods and record the running times for initial data of different levels of regularity. The $H^1$-errors versus CPU time of all the methods are presented in Figure \ref{fig:rough_data_cpu}. {Although the $H^1$-errors of SLRI and RLRI2-v are similar in Figure \ref{fig:rough_data_error}, all RLRIs-v are more efficient than the SLRI under the same regularity initial condition.} We note that the Strang splitting performs best for smooth initial data.

Subsequently, to verify the Assumption \ref{Assumption2} (i), which plays a crucial role in the analysis, we test the evolutions of the relaxation coefficients \(\gamma_n\) with \(\tau=0.01\) and \(d(\gamma_n)\) as a function of \(\tau\) for low-regularity and smooth initial data.
The numerical results are plotted in Figure \ref{fig:H2_gamma}, from which it is clearly observed that \(\gamma_n\) fluctuates around \(1\) and the order of \(d(\gamma_n)\) is approximately \(1\). {In addition, \(\gamma_n\) of RLRI-u is closer to \(1\) than \(\gamma_n\) of RLRIs-v.}

Finally, the relative errors in the \(L^2\)-norm are examined up to \(T = 5000\) with \(K = 2^{10}\) and \(\tau = 0.02\) for initial data \(u_0 \in H^2,\, H^3,\) and \(C^\infty\), as shown in Figure~\ref{fig:H2_normerr}.  
It is evident that all newly proposed methods outperform the classical schemes.  
The errors of RLRIs-v and RLRI-u reach the level of machine precision (around \(10^{-16}\)), and notably, do not grow over time, highlighting the excellent long-time performance of the proposed methods. On the other hand, the error curves appear discontinuous due to the disappearance of points where the relative \(L^2\)-norm errors are numerically zero.  
This indicates that the new methods are superior in both theoretical analysis and numerical performance.  
To further substantiate this, Table~\ref{stepwise_norm_error} reports the maximum and minimum of time-stepwise relative errors in the \(L^2\)-norm for all three types of initial data.



\begin{figure}[htbp]
	\centering
	
	\subfigure[\(\gamma_n\) with \(u_0\in H^2\)]{
		\includegraphics[width=0.32\textwidth, height=3.9cm]{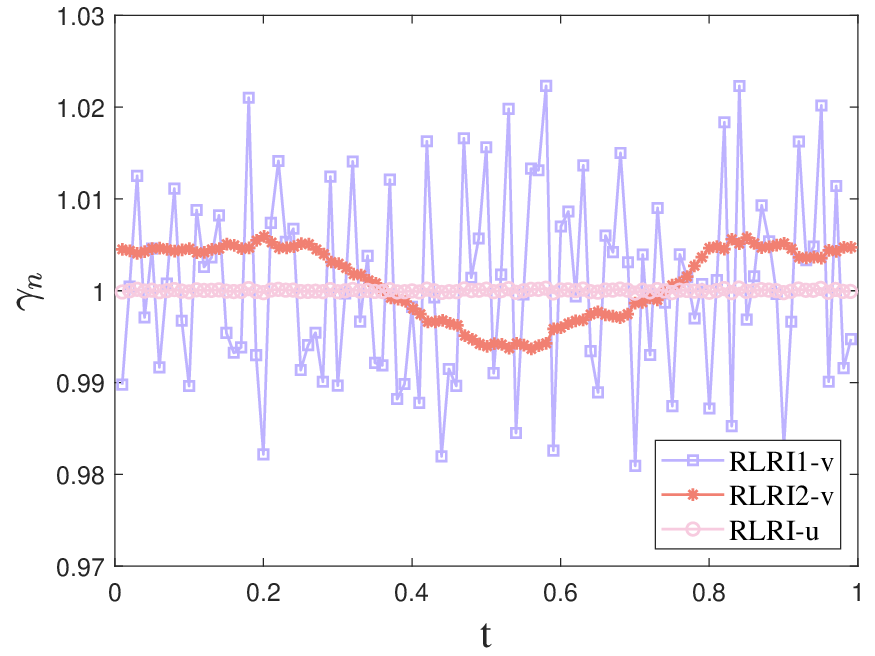}
	}
				\hspace{-1em}
\subfigure[\(\gamma_n\) with \(u_0\in H^3\)]{
		\includegraphics[width=0.32\textwidth, height=3.9cm]{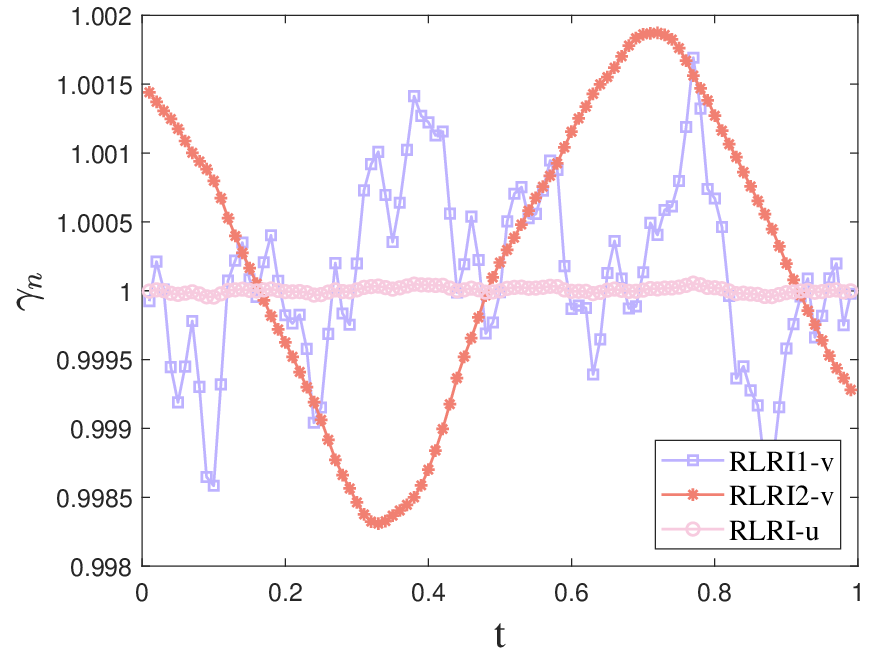}
	}
	\hspace{-1em}
	\subfigure[\(\gamma_n\) with \(u_0\in C^\infty\)]{
		\includegraphics[width=0.32\textwidth, height=3.9cm]{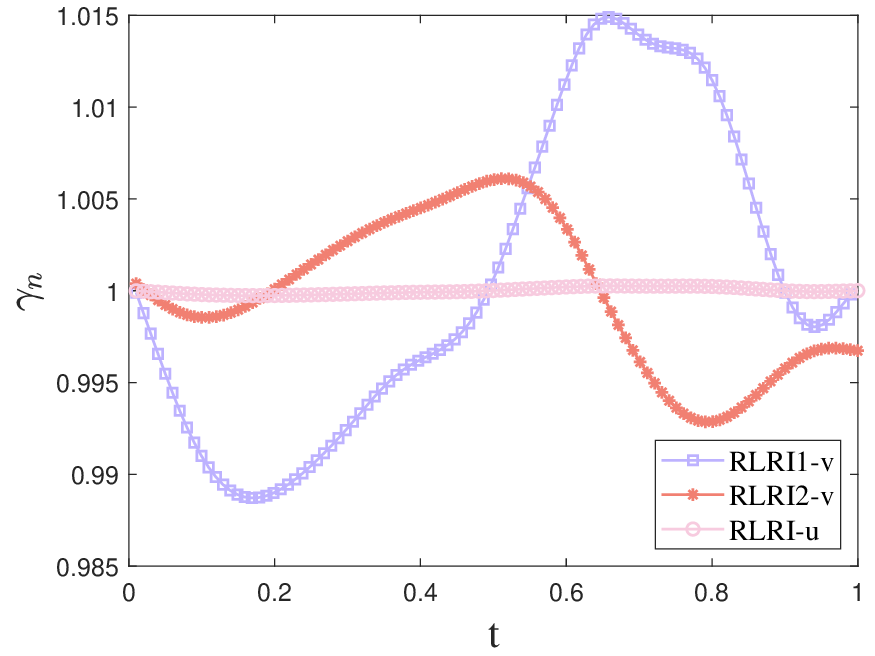}
	}
\hspace{-1em}
	\subfigure[\(d(\gamma_n)\)  with \(u_0\in H^2\)]{
		\includegraphics[width=0.32\textwidth, height=3.9cm]{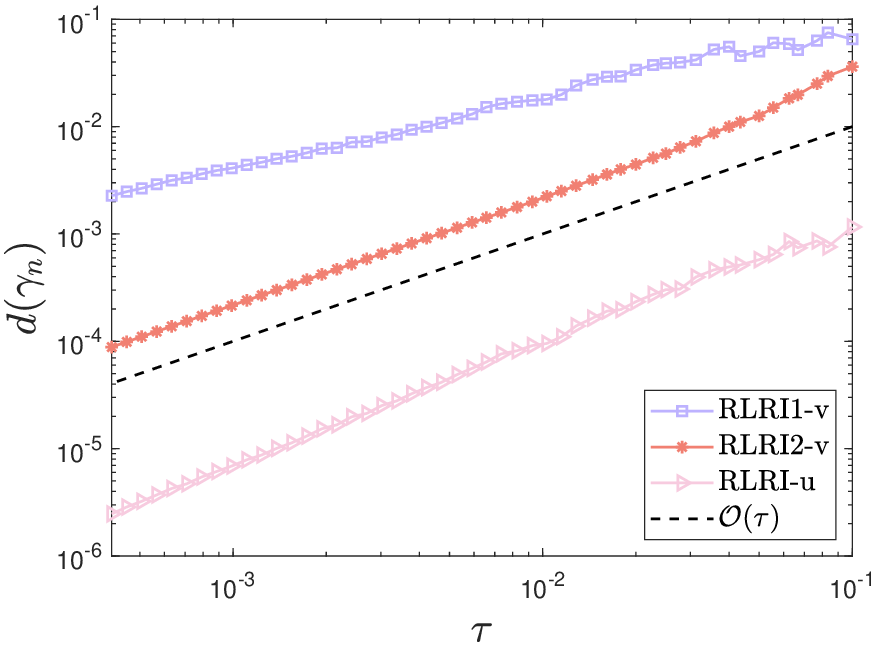}
	}
	\hspace{-1em}
	\subfigure[\(d(\gamma_n)\) with \(u_0\in H^3\)]{
		\includegraphics[width=0.32\textwidth, height=3.9cm]{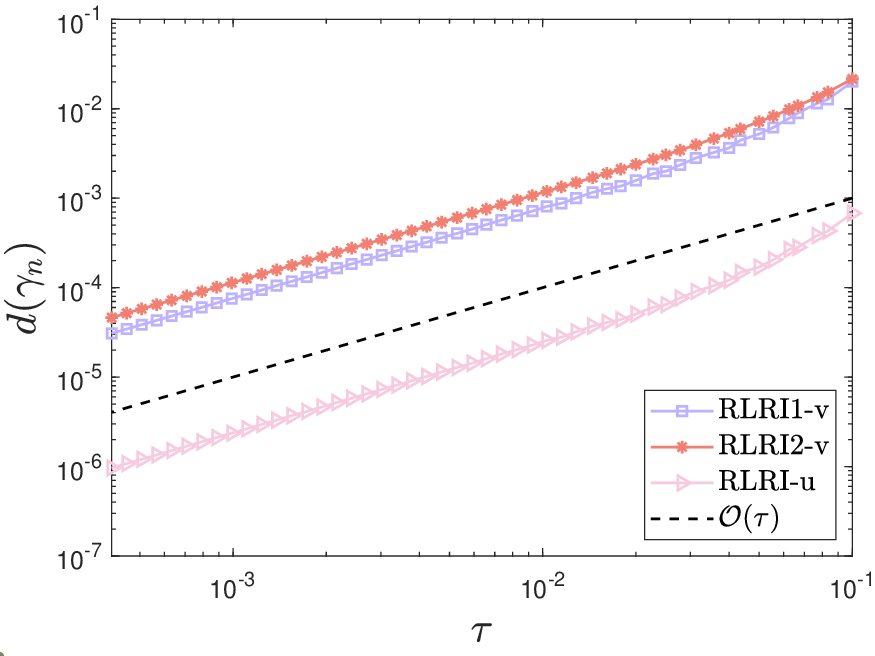}
	}
	\hspace{-1em}
	\subfigure[\(d(\gamma_n)\)  with \(u_0\in C^\infty\)]{
		\includegraphics[width=0.32\textwidth, height=3.9cm]{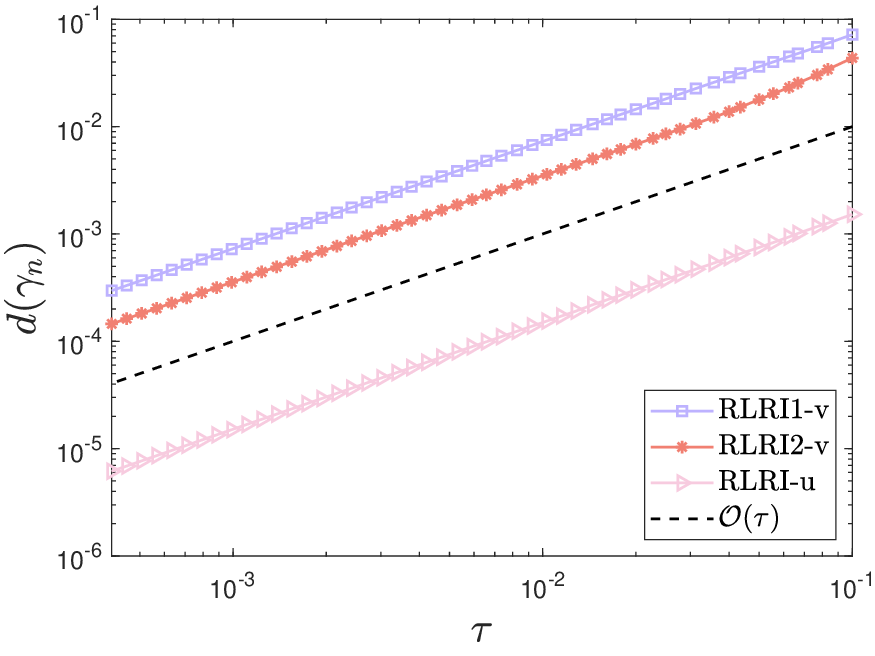}
	}
	\caption{Evolutions of the relaxation coefficients \(\gamma_n\)   with \(\tau=0.01\) and \(d(\gamma_n)\)  at \(T=1\) as a function of \(\tau\) for different initial data.}
	\label{fig:H2_gamma}
\end{figure}

%
%

%
%

\begin{figure}[htbp]
	\centering
	
	\subfigure[\(u_0\in H^2\), \(\theta=2\) in 
	\eqref{rough_data} ]{
		\includegraphics[width=0.32\textwidth, height=3.9cm]{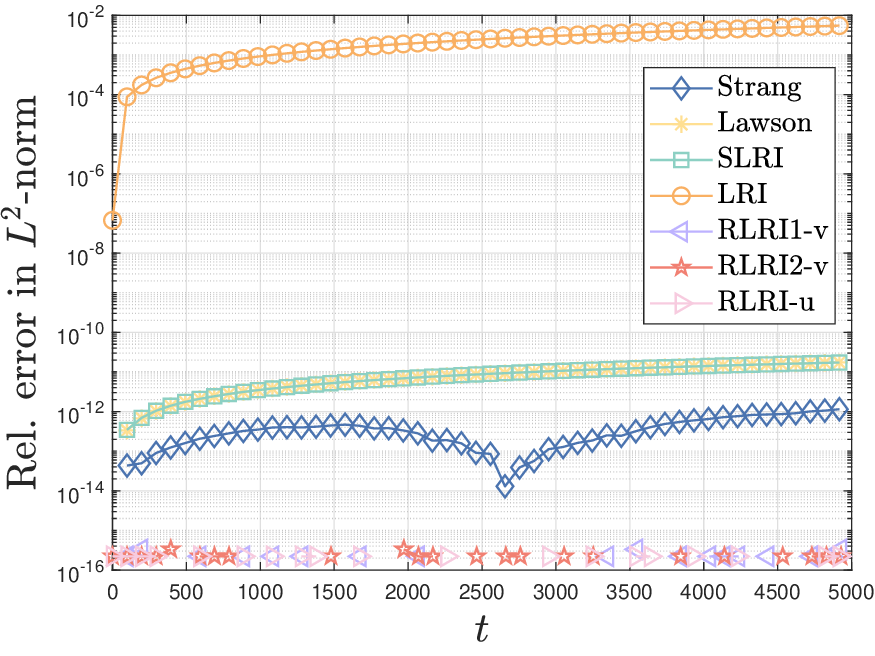}
	}
	\hspace{-1em}
	\subfigure[\(u_0\in H^3\), \(\theta=3\) in 
	\eqref{rough_data}  ]{
		\includegraphics[width=0.32\textwidth, height=3.9cm]{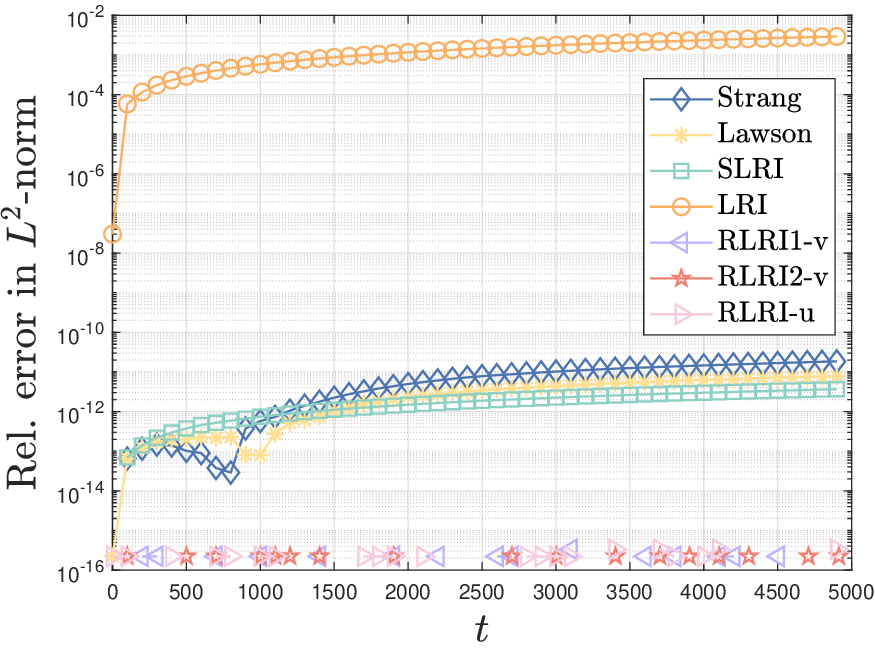}
	}
		\hspace{-1em}
	\subfigure[\(u_0 \in C^\infty\)  in 
\eqref{smooth_initial_data}  ]{
		\includegraphics[width=0.32\textwidth, height=3.9cm]{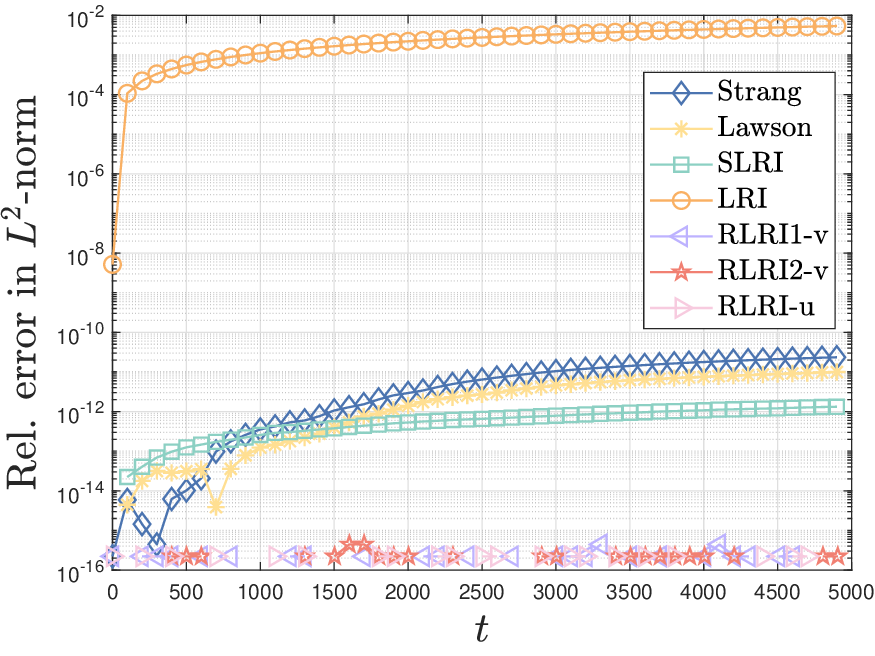}
	}
	
	\caption{Relative error in the \(L^2\)-norm for \(\tau=0.02\) and different initial data.}
	\label{fig:H2_normerr}
\end{figure}

%
%
\begin{table}[htbp]
	\centering
	\caption{Time stepwise relative error in the \(L^2\)-norm for \(\tau=0.02\) and \(T=5000\).}
	{\setlength{\tabcolsep}{3pt}\begin{tabular}{ccccccccc}
			\toprule
			$u_0\in H^2$ & $\max$ & $\min$ & $u_0\in H^3$ & $\max$ & $\min$ & $u_0\in C^\infty$ & $\max$ & $\min$ \\
			\midrule
			RLRI1-v& 4.44e-16&0&
			RLRI1-v&4.44e-16&0&
			RLRI1-v&6.66e-16&0\\
			RLRI2-v& 4.44e-16&0&
			RLRI2-v& 4.44e-16&0&
			RLRI2-v&4.44e-16&0\\
			RLRI-u&3.33e-16&0&
			RLRI-u&4.44e-16&0&
			RLRI-u&4.44e-16&0\\
			\bottomrule
	\end{tabular}}
	\label{stepwise_norm_error}
\end{table}

\section{Conclusion}\label{con}
In this work, we have introduced a novel class of fully explicit structure-preserving low-regularity exponential integrators for the nonlinear Schrödinger equation, constructed via a time-relaxation strategy in the twisted variable. The methods enforce exact mass conservation through an adaptive relaxation parameter while retaining optimal convergence under low-regularity assumptions. The proposed framework draws inspiration from time-relaxation technique in energy-preserving ODE schemes and successfully extends them to the low-regularity PDE setting. Our general and flexible approach enables application to a broad class of dispersive semilinear evolution equations beyond the NLS and supports future development of structure-preserving extensions.


\bibliographystyle{siamplain}
\bibliography{references}

\begin{thebibliography}{10}

\bibitem{akrivis1993finite}
{\sc G.~D. Akrivis}, {\em {Finite difference discretization of the cubic
  Schr{\"o}dinger equation}}, IMA J. Numer. Anal., 13 (1993), pp.~115--124.

\bibitem{bronsard2023error}
{\sc Y.~Alama~Bronsard}, {\em {Error analysis of a class of semi-discrete
  schemes for solving the Gross--Pitaevskii equation at low regularity}}, J.
  Comput. Appl. Math., 418 (2023), p.~114632.

\bibitem{alama2023symmetric}
{\sc Y.~Alama~Bronsard}, {\em {A symmetric low-regularity integrator for the
  nonlinear Schr{\"o}dinger equation}}, IMA J. Numer. Anal., 44 (2024),
  pp.~3648--3682.

\bibitem{antoine2013computational}
{\sc X.~Antoine, W.~Bao, and C.~Besse}, {\em {Computational methods for the
  dynamics of the nonlinear Schr{\"o}dinger/Gross--Pitaevskii equations}},
  Comput. Phys. Commun., 184 (2013), pp.~2621--2633.

\bibitem{bai2022constructive}
{\sc G.~Bai, B.~Li, and Y.~Wu}, {\em {A constructive low-regularity integrator
  for the one-dimensional cubic nonlinear Schr{\"o}dinger equation under the
  Neumann boundary condition}}, IMA J. Numer. Anal., 43 (2023), pp.~3243--3281.

\bibitem{bao2012mathematical}
{\sc W.~Bao and Y.~Cai}, {\em {Mathematical theory and numerical methods for
  Bose--Einstein condensation}}, Kinet. Relat. Models, 6 (2013), pp.~1--135.

\bibitem{bao2013optimal}
{\sc W.~Bao and Y.~Cai}, {\em {Optimal error estimates of finite difference
  methods for the Gross--Pitaevskii equation with angular momentum rotation}},
  Math. Comp., 82 (2013), pp.~99--128.

\bibitem{bao2014uniform}
{\sc W.~Bao and Y.~Cai}, {\em {Uniform and optimal error estimates of an
  exponential wave integrator sine pseudospectral method for the nonlinear
  Schrödinger equation with wave operator}}, SIAM J. Numer. Anal., 52 (2014),
  pp.~1103--1127.

\bibitem{bao2023improved}
{\sc W.~Bao, Y.~Cai, and Y.~Feng}, {\em {Improved uniform error bounds of the
  time-splitting methods for the long-time (nonlinear) Schr{\"o}dinger
  equation}}, Math. Comp., 92 (2023), pp.~1109--1139.

\bibitem{bao2003numerical}
{\sc W.~Bao, D.~Jaksch, and P.~A. Markowich}, {\em {Numerical solution of the
  Gross--Pitaevskii equation for Bose--Einstein condensation}}, J. Comput.
  Phys., 187 (2003), pp.~318--342.

\bibitem{bao2024explicit}
{\sc W.~Bao and C.~Wang}, {\em {An explicit and symmetric exponential wave
  integrator for the nonlinear {S}chr{\"o}dinger equation with low regularity
  potential and nonlinearity}}, SIAM J. Numer. Anal., 62 (2024),
  pp.~1901--1928.

\bibitem{besse2002order}
{\sc C.~Besse, B.~Bid{\'e}garay, and S.~Descombes}, {\em {Order estimates in
  time of splitting methods for the nonlinear Schr{\"o}dinger equation}}, SIAM
  J. Numer. Anal., 40 (2002), pp.~26--40.

\bibitem{bruned_schratz_2022}
{\sc Y.~Bruned and K.~Schratz}, {\em Resonance-based schemes for dispersive
  equations via decorated trees}, Forum Math. Pi, 10 (2022), pp.~e2 1--76.

\bibitem{cao2024new}
{\sc J.~Cao, B.~Li, and Y.~Lin}, {\em {A new second-order low-regularity
  integrator for the cubic nonlinear Schr{\"o}dinger equation}}, IMA J. Numer.
  Anal., 44 (2024), pp.~1313--1345.

\bibitem{Celledoni2008}
{\sc E.~Celledoni, D.~Cohen, and B.~Owren}, {\em {Symmetric exponential
  integrators with an application to the cubic Schr{\"o}dinger equation}},
  Found. Comput. Math., 8 (2008), pp.~303--317.

\bibitem{erdhos2007derivation}
{\sc L.~Erd\H{o}s, B.~Schlein, and H.~T. Yau}, {\em Derivation of the cubic
  non-linear {S}chr\"{o}dinger equation from quantum dynamics of many-body
  systems}, Invent. Math., 167 (2007), pp.~515--614.

\bibitem{feng_maierhofer_schratz_2023}
{\sc Y.~Feng, G.~Maierhofer, and K.~Schratz}, {\em Long-time error bounds of
  low-regularity integrators for nonlinear {S}chr\"odinger equations}, Math.
  Comp., 93 (2023), pp.~1569--1598.

\bibitem{feng2025explicit}
{\sc Y.~Feng, G.~Maierhofer, and C.~Wang}, {\em Explicit symmetric
  low-regularity integrators for the nonlinear {S}chr{\"o}dinger equation},
  SIAM J. Sci. Comput., 47 (2025), pp.~A2154--A2179.

\bibitem{hochbruck2010exponential}
{\sc M.~Hochbruck and A.~Ostermann}, {\em Exponential integrators}, Acta
  Numer., 19 (2010), pp.~209--286.

\bibitem{jahnke2023numerical}
{\sc T.~Jahnke and M.~Kirn}, {\em {On numerical methods for the
  semi-nonrelativistic limit system of the nonlinear {D}irac equation}}, BIT
  Numer. Math., 63 (2023), p.~26.

\bibitem{jahnke2018adiabatic}
{\sc T.~Jahnke and M.~Mikl}, {\em {Adiabatic midpoint rule for the
  dispersion-managed nonlinear {S}chr{\"o}dinger equation}}, Numer. Math., 138
  (2018), pp.~975--1009.

\bibitem{jahnke2019adiabatic}
{\sc T.~Jahnke and M.~Mikl}, {\em {Adiabatic exponential midpoint rule for the
  dispersion-managed nonlinear {S}chr{\"o}dinger equation}}, IMA J. Numer.
  Anal., 39 (2019), pp.~1818--1859.

\bibitem{Ketcheson2019}
{\sc D.~I. Ketcheson}, {\em Relaxation {R}unge-{K}utta methods: conservation
  and stability for inner-product norms}, SIAM J. Numer. Anal., 57 (2019),
  pp.~2850--2870.

\bibitem{Ranocha2020}
{\sc D.~I. Ketcheson}, {\em Relaxation {R}unge-{K}utta methods for
  {H}amiltonian problems}, J. Sci. Comput., 84 (2020), p.~17.

\bibitem{Kn19}
{\sc M.~Kn\"{o}ller, A.~Ostermann, and K.~Schratz}, {\em A fourier integrator
  for the cubic nonlinear {S}chr\"{o}dinger equation with rough initial data},
  SIAM J. Numer. Anal., 57 (2019), pp.~1976--1986.

\bibitem{li2021fully}
{\sc B.~Li and Y.~Wu}, {\em {A fully discrete low-regularity integrator for the
  1D periodic cubic nonlinear Schr{\"o}dinger equation}}, Numer. Math., 149
  (2021), pp.~151--183.

\bibitem{Lubich2008}
{\sc C.~Lubich}, {\em {On splitting methods for {Schr{\"o}dinger-Poisson} and
  cubic nonlinear {Schr{\"o}dinger} equations}}, Math. Comp., 77 (2008),
  pp.~2141--2153.

\bibitem{Maierhofer2022}
{\sc G.~Maierhofer and K.~Schratz}, {\em {Bridging the gap: symplecticity and
  low regularity in Runge-Kutta resonance-based schemes}}, Math. Comp.,
  (2025), \url{https://doi.org/10.1090/mcom/4105}.

\bibitem{ostermann_schratz2018low}
{\sc A.~Ostermann and K.~Schratz}, {\em {Low regularity exponential-type
  integrators for semilinear Schr{\"o}dinger equations}}, Found. Comput. Math.,
  18 (2018), pp.~731--755.

\bibitem{ostermann2022second}
{\sc A.~Ostermann, Y.~Wu, and F.~Yao}, {\em {A second-order low-regularity
  integrator for the nonlinear Schr{\"o}dinger equation}}, Adv. Contin.
  Discrete Models, 2022 (2022), p.~23.

\bibitem{ostermann2022fully}
{\sc A.~Ostermann and F.~Yao}, {\em {A fully discrete low-regularity integrator
  for the nonlinear Schr{\"o}dinger equation}}, J. Sci. Comput., 91 (2022),
  p.~9.

\bibitem{Ranocha2022}
{\sc H.~Ranocha, M.~Sayyari, L.~Dalcin, M.~Parsani, and D.~I. Ketcheson}, {\em
  Relaxation {R}unge-{K}utta methods: fully-discrete explicit entropy-stable
  schemes for the compressible {E}uler and {N}avier-{S}tokes equations}, SIAM
  J. Sci. Comput., 42 (2020), pp.~A612--638.

\bibitem{rousset2021general}
{\sc F.~Rousset and K.~Schratz}, {\em A general framework of low regularity
  integrators}, SIAM J. Numer. Anal., 59 (2021), pp.~1735--1768.

\bibitem{sulem1999nonlinear}
{\sc C.~Sulem and P.~L. Sulem}, {\em {The Nonlinear {S}chr\"{o}dinger Equation:
  Self-Focusing and Wave Collapse}}, Springer, 1999.

\bibitem{tai1986observation}
{\sc K.~Tai, A.~Hasegawa, and A.~Tomita}, {\em Observation of modulational
  instability in optical fibers}, Phys. Rev. Lett., 56 (1986), pp.~135--139.

\bibitem{thomas2012nonlinear}
{\sc R.~Thomas, C.~Kharif, and M.~Manna}, {\em A nonlinear {S}chr{\"o}dinger
  equation for water waves on finite depth with constant vorticity}, Phys.
  Fluids, 24 (2012), p.~127102.

\end{thebibliography}

\end{document}